\documentclass{article}

     \PassOptionsToPackage{numbers, compress}{natbib}

     \usepackage[preprint]{neurips-2019}

\usepackage[utf8]{inputenc} 
\usepackage[T1]{fontenc}    
\usepackage{hyperref}       
\usepackage{url}            
\usepackage{booktabs}       
\usepackage{amsfonts}       
\usepackage{nicefrac}       
\usepackage{microtype}      

\usepackage{amssymb,amsthm,amsmath}
\usepackage{mathtools}
\usepackage{algorithm}
\usepackage{algpseudocode}

\usepackage{bbm}
\usepackage{tikz}
\usetikzlibrary{shapes.misc}
\usepackage{graphicx}
\usepackage[]{caption}
\usepackage{subcaption}
\usepackage{pgfplots}
\usepgfplotslibrary{fillbetween}
\pgfplotsset{compat=1.14} 
\allowdisplaybreaks
\newtheorem{theorem}{Theorem}
\newtheorem{lemma}{Lemma}
\newtheorem{proposition}{Proposition}
\newtheorem{corollary}{Corollary}
\newtheorem{definition}{Definition}
\newtheorem{assumption}{Assumption}

\usepackage{float}
\makeatletter
\newenvironment{breakablealgorithm}
  {
   \begin{center}
     \refstepcounter{algorithm}
     \hrule height.8pt depth0pt \kern2pt
     \renewcommand{\caption}[2][\relax]{
       {\raggedright\textbf{\ALG@name~\thealgorithm} ##2\par}%
       \ifx\relax##1\relax 
         \addcontentsline{loa}{algorithm}{\protect\numberline{\thealgorithm}##2}%
       \else 
         \addcontentsline{loa}{algorithm}{\protect\numberline{\thealgorithm}##1}%
       \fi
       \kern2pt\hrule\kern2pt
     }
  }{
     \kern2pt\hrule\relax
   \end{center}
  }
\makeatother

\usepackage{multirow}
\usepackage{placeins}
\usepackage{todonotes}

\usepackage{xcolor,colortbl}
\definecolor{Gray}{gray}{0.9}
\newcolumntype{a}{>{\columncolor{Gray}}c}

\title{Sparse High-Dimensional Isotonic Regression}

\author{%
   David Gamarnik \thanks{http://web.mit.edu/gamarnik/www/home.html} \\
  Sloan School of Management\\
  Massachusetts Institute of Technology \\
  Cambridge, MA 02139\\
  \texttt{gamarnik@mit.edu} \\
  \And
   Julia Gaudio\thanks{http://web.mit.edu/jgaudio/www/index.html} \\
  Operations Research Center\\
  Massachusetts Institute of Technology\\
  Cambridge, MA 02139 \\
  \texttt{jgaudio@mit.edu} \\
}

\begin{document}

\maketitle

\begin{abstract}
We consider the problem of estimating an unknown coordinate-wise monotone function given noisy measurements, known as the isotonic regression problem. Often, only a small subset of the features affects the output. This motivates the \emph{sparse} isotonic regression setting, which we consider here. We provide an upper bound on the expected VC entropy of the space of sparse coordinate-wise monotone functions, and identify the regime of statistical consistency of our estimator. We also propose a linear program to recover the active coordinates, and provide theoretical recovery guarantees. We close with experiments on cancer classification, and show that our method significantly outperforms standard methods.
\end{abstract}

\section{Introduction}
Given a partial order $\preceq$ on $\mathbb{R}^d$, we say that a function $f : \mathbb{R}^d \to \mathbb{R}$ is \emph{monotone} if for all $x_1, x_2 \in \mathbb{R}^d$ such that $x_1 \preceq x_2$, it holds that $f(x_1) \leq f(x_2)$. In this paper, we study the univariate isotonic regression problem under the standard Euclidean partial order. Namely, we define the partial order $\preceq$ on $\mathbb{R}^d$ as follows: $x_1 \preceq x_2$ if $x_{1,i} \leq x_{2,i}$ for all $i \in \{1, \dots, d\}$. If $f$ is monotone according to the Euclidean partial order, we say $f$ is \emph{coordinate-wise monotone}. 

This paper introduces the \emph{sparse isotonic regression problem}, defined as follows. 
Write $x_1 \preceq_{A} x_2$ if $x_{1,i} \leq x_{2,i}$ for all $i \in A$. We say that a function $f$ on $\mathbb{R}^d$ is \emph{$s$-sparse coordinate-wise monotone} if for some set $A \subseteq [d]$ with $|A| = s$, it holds that $x_1 \preceq_{A} x_2 \implies f(x_1) \leq f(x_2)$. We call $A$ the \emph{set of active coordinates}. The sparse isotonic problem is to estimate the $s$-sparse coordinate-wise function $f$ from samples, knowing the sparsity level $s$ but not the set $A$. 
We consider two different noise models. In the Noisy Output Model, the input $X$ is a random variable supported on $[0,1]^d$, and $W$ is zero-mean noise that is independent from $X$. The model is $Y = f(X) + W$. We assume that $Y \in [0,1]$ almost surely. In the Noisy Input Model, $Y = f(X+ W)$, and we exclusively consider the classification problem, namely $f : \mathbb{R}^d \to \{0,1\}$. In either noise model, we assume that $n$ independent samples $(X_1, Y_1), \dots, (X_n, Y_n)$ are given. 

The goal of our paper is to produce an estimator $\hat{f}_n$ and give statistical guarantees for it. To our knowledge, the only work that provides statistical guarantees on isotonic regression estimators in the Euclidean partial order setting with $d \geq 3$ is the paper of Han et al (\cite{Han2017}). The authors give guarantees of the empirical $L_2$ loss, defined as $R(\hat{f}_n, f_0) = \mathbb{E}\left[ \frac{1}{n} \sum_{i=1}^n \left( \hat{f}_n(X_i) - f_0(X_i)\right)^2 \right]$, where the expectation is over the samples $X_1, \dots X_n$. 
In this paper, we expand on the work of Gamarnik (\cite{Gamarnik1999}), to the high-dimensional sparse setting. It is shown in \cite{Gamarnik1999} that the expected Vapnik-Chervonenkis entropy of the class of coordinate-wise monotone functions grows subexponentially. The main result of \cite{Gamarnik1999} is that when $X \in [0,1]^2$ and $Y \in [0,1]$ almost surely, \[\mathbb{P} \left( \Vert \hat{f}_n - f \Vert_2^2 > \epsilon \right) \leq e^{\left \lceil \frac{4}{\epsilon} \right \rceil \sqrt{n}-\frac{\epsilon^2 n}{256}},\] where $\hat{f}_n$ is a coordinate-wise monotone funtion, estimated based on empirical mean squared error. This result shows that the estimated function converges to the true function in $L_2$, almost surely (\cite{Gamarnik1999}). In this paper, we extend the work of \cite{Gamarnik1999} to the sparse high-dimensional setting, where the problem dimension $d$ and the sparsity $s$ may diverge to infinity as the sample size $n$ goes to infinity.

We propose two algorithms for the estimation of the unknown $s$-sparse coordinate-wise monotone function $f$. The \emph{simultaneous} algorithm determines the active coordinates and the estimated function values in a single optimization formulation. The \emph{two-stage} algorithm first determines the active coordinates via a linear program, and then estimates function values. The sparsity level is treated as constant or moderately growing. We give statistical consistency and support recovery guarantees for the Noisy Output Model, analyzing both the simultaneous and two-stage algorithms. We show that when $n = e^{\omega(s^2)}$ and $n = \omega\left(s \log(d)\right)$, the estimator $\hat{f}_n$ from the simultaneous procedure is statistically consistent. In particular, when the sparsity level is constant, the dimension can be much larger than the sample size. To analyze the two-stage approach, we show that if $n = \omega( \log(d))$ and $s$ is constant, then the linear program correctly recovers the support with high probability. We also give statistical consistency guarantees for the simultaneous and two-stage algorithms in the Noisy Input Model, assuming that the components of $W$ are independent. We show that in the regime where $s$ is constant and $n = \omega(\log(d))$, the estimators from both algorithms are consistent.

The isotonic regression problem has a long history in the statistics literature; see for example the books \cite{Barlow1973} and \cite{Robertson1988}. The emphasis of most research in the area of isotonic regression has been the design of algorithms: for example, the Pool Adjacent Violators algorithm (\cite{Kruskal1964}), active set methods (\cite{Best1990}, \cite{deLeeuw2009}), and the Isotonic Recursive Partitioning algorithm (\cite{Luss2012}). In addition to the univariate setting ($f: \mathbb{R}^d \to \mathbb{R}$), the multivariate setting ($f: \mathbb{R}^d \to \mathbb{R}^q$, $q \geq 2$) has also been considered; see e.g. \cite{Sasabuchi1983} and \cite{Sasabuchi1992}. In the multivariate setting, whenever $x_1 \preceq x_2$ according to some defined partial order $\preceq$, it holds that $f(x_1) \tilde{\preceq} f(x_2)$, where $\tilde{\preceq}$ is some other defined partial order. There are many applications for the coordinate-wise isotonic regression problem. For example, Dykstra and Robertson (1982) showed that isotonic regression could be used to predict college GPA from standardized test scores and high school GPA\nocite{Dykstra1982}. Luss et al (2012) applied isotonic regression to the prediction of baseball players' salaries, from the number of runs batted in and the number of hits\nocite{Luss2012}. Isotonic regression has found rich applications in biology and medicine, particularly to build disease models (\cite{Luss2012}, \cite{Schell1997}). 

The rest of the paper is structured as follows. Section \ref{sec:algorithms} gives the simultaneous and two-stage algorithms for sparse isotonic regression. Sections \ref{sec:noisy-output} and \ref{sec:noisy-input} provide statistical consistency and recovery guarantees for the Noisy Output and Noisy Input models. All proofs can be found in the supplementary material. In Section \ref{sec:experiments}, we provide experimental evidence for the applicability of our algorithms. 
We test our algorithm on a cancer classification task, using gene expression data. Our algorithm achieves a success rate of about $96\%$ on this task, significantly outperforming the $k$-Nearest Neighbors classifier and the Support Vector Machine.

\section{Algorithms for sparse isotonic regression}\label{sec:algorithms}
In this section, we present our two algorithmic approaches for sparse isotonic regression: the simultaneous and two-stage algorithms. Let $\mathcal{R}$ be the range of $f$. In the Noisy Output Model, $\mathcal{R} \subseteq [0,1]$, and in the Noisy Input Model, $\mathcal{R} = \{0, 1\}$.
\subsection{The Simultaneous Algorithm}
The simultaneous algorithm solves the following problem.
\begin{align}
&\min_{A, F} \sum_{i=1}^n \left(Y_i - F_i\right)^2 \label{eq:optimization-first}\\
\text{s.t. }&|A| = s \\
&F_i \leq F_j &\text{ if } X_i \preceq_A X_j\\
&F_i \in \mathcal{R} &\forall i \label{eq:optimization-last}
\end{align}
The estimated function $\hat{f}_n$ is determined by interpolating from the pairs $(X_1, F_1), \dots, (X_n, F_n)$ in a straightforward way. In particular, $\hat{f}_n(x) = \min\{y \in \mathcal{R} : X_i \preceq x \implies y \geq F_i\}$. We call this the ``min'' interpolation rule. The ``max'' interpolation rule is $\hat{f}_n(x) =  \max\{y \in \mathcal{R} : x \preceq X_i \implies y \leq F_i\}$. 
\begin{definition}
For inputs $X_1, \dots, X_n$, let $q(i,j,k) = 1$ if $X_{i,k} > X_{j,k}$, and $q(i,j,k) = 0$ otherwise.
\end{definition}
Problem \eqref{eq:optimization-first}-\eqref{eq:optimization-last} can be encoded as a single mixed-integer convex minimization. We refer to the resulting Algorithm \ref{alg:simultaneous} as Integer Programming Isotonic Regression (IPIR). The details of the algorithm are found in the Appendix.

\subsection{The Two-Stage Algorithm}
The two-stage algorithm estimates the active coordinates through a linear program, using these to then estimate the function values. The process of estimating the active coordinates is referred to as \emph{support recovery}. The active coordinates may be estimated all at once (Algorithm \ref{alg:simultaneous-recovery}) or sequentially (Algorithm \ref{alg:sequential-recovery}). Algorithm \ref{alg:simultaneous-recovery} is referred to as Linear Programming Support Recovery (LPSR) and Algorithm \ref{alg:sequential-recovery} is referred to as Sequential Linear Programming Support Recovery (S-LPSR). The details of the algorithms are given in the Appendix. 
The two-stage algorithm for estimating $\hat{f}_n$ first estimates the set of active coordinates using the LPSR or S-LPSR algorithm, and then estimates the function values. The results algorithm is referred to as Two Stage Isotonic Regression (TSIR), and is given in the Appendix (Algorithm \ref{alg:two-stage}).

\section{Results on the Noisy Output Model}\label{sec:noisy-output}
Recall the Noisy Output Model: $Y = f(X) + W$, where $f$ is an $s$-sparse coordinate-wise monotone function with active coordinates $A$. We assume throughout this section that $X$ is a uniform random variable on $[0,1]^d$, $W$ is a zero-mean random variable independent from $X$, and the domain of $f$ is $[0,1]^d$. We additionally assume that $Y \in [0,1]$ almost surely. Up to shifting and scaling, this is equivalent to assuming that $f$ has a bounded range and $W$ has a bounded support.
 
\subsection{Statistical consistency}\label{subsec:noisy-output-consistency}
In this section, we extend the results of \cite{Gamarnik1999}, in order to demonstrate the statistical consistency of the estimator produced by Algorithm \ref{alg:simultaneous}. The consistency will be stated in terms of the $L_2$ norm error.

\begin{definition}[$L_2$ Norm Error]
For an estimator $\hat{f}_n$, define 
\begin{align*}
\Vert \hat{f}_n - f \Vert_2^2 &\triangleq \int_{x \in [0, 1]^d} \left(\hat{f}_n(x) - f(x) \right)^2 dx.
\end{align*}
We call $\Vert \hat{f}_n - f \Vert_2$ the $L_2$ norm error. 
\end{definition}

\begin{definition}[Consistent Estimator]
Let $\hat{f}_n$ be a estimator for the function $f$. We say that $\hat{f}_n$ is \emph{consistent} if for all $\epsilon > 0$, it holds that
\[\lim_{n \to \infty} \mathbb{P} \left( \Vert \hat{f}_n - f \Vert_2 \geq \epsilon \right) \to 0.\]
\end{definition}
 
\begin{theorem}\label{thm:statistical-consistency} 
The $L_2$ error of the estimator $\hat{f}_n$ obtained from Algorithm \ref{alg:simultaneous} is upper bounded as
\[\mathbb{P}\left( \Vert \hat{f}_n - f \Vert_2 \geq \epsilon \right) \leq 6 \binom{d}{s} \exp\left\{\left(\left \lceil \frac{2^{11}}{\epsilon^2} \right \rceil - 1 \right) \left(2^s + 2\log(2) -1 \right) n^{\frac{s-1}{s}}  {-\frac{3 \epsilon^3 n}{41 \times 2^{10}}} \right\}.\] 
\end{theorem}

\begin{corollary}\label{corollary:noisy-output-statistical-consistency}
When $n = e^{\omega(s^2)}$ and $n = \omega(s \log(d))$, the estimator $\hat{f}_n$ from Algorithm \ref{alg:simultaneous} is consistent. Namely, $\Vert \hat{f}_n - f \Vert_2 \to 0$ in probability as $n \to \infty$. In particular, if the sparsity level is constant, the sample complexity is only logarithmic in the dimension.
\end{corollary}

\subsection{Support recovery}\label{subsec:noisy-output-support-recovery}
In this subsection, we give support recovery guarantees for Algorithm \ref{alg:sequential-recovery}. The guarantees will be in terms of differences of probabilities.
\begin{definition}
Let $Y_1 = f(X_1) + W_1$ and $Y_2 = f(X_2) + W_2$ be two independent samples from the model. For $k \in A$, let 
\[p_{k} \triangleq \mathbb{P} \left( Y_1 > Y_2~|~q(1,2,k) = 1 \right) - \mathbb{P} \left( Y_1 < Y_2~|~q(1,2,k) = 1 \right) .\]
Assume without loss of generality that $A = \{1, 2, \dots, s\}$ and $p_1 \leq p_2 \leq \dots \leq p_s$. 
\end{definition}

\begin{theorem}\label{thm:multiple-coordinates}
Let $B$ be the set of indices corresponding to running Algorithm \ref{alg:sequential-recovery} using $N = s \cdot n$ samples. Then it holds that $B = A$ with probability at least
 \[1 - (d-s)\sum_{k=1}^s (s+1 -k) \exp\left(-\frac{n p_{k}^2}{16} \right). \]
Therefore, Algorithm \ref{alg:sequential-recovery} recovers the true active coordinates with the above probability using $N$ samples. 
\end{theorem}

\begin{corollary}\label{corollary:recovery}
Assume that $p_1 = \Theta(1)$. Let $N = s \cdot n$ be the number of samples used by Algorithm \ref{alg:sequential-recovery}. If $N = s \cdot \omega( \log(d))$, then Algorithm \ref{alg:sequential-recovery} recovers the true support w.h.p. as $s, d \to \infty$. 
\end{corollary}
Note that if $s$ is itself constant, then $p_1 = \Theta(1)$.

We can now give a guarantee of the success of Algorithm \ref{alg:two-stage}, using Algorithm \ref{alg:sequential-recovery} for support recovery.
\begin{corollary}\label{corollary:two-stage-noisy-output}
Assume that $p_1 = \Theta(1)$. Consider running Algorithm \ref{alg:two-stage} using $s \cdot n$ samples for sequential recovery and an additional $n$ samples for function value estimation. Let $N = (s+1) n$ be the total sample size, and let $\hat{f}_N$ be the estimated function. If $N = s \cdot \omega(\log(d))$ and $N = s e^{\omega(s^2)}$, then $\hat{f}_N$ is a consistent estimator. 
\end{corollary}
Corollary \ref{corollary:two-stage-noisy-output} shows that if $s$ is constant, then Algorithm \ref{alg:two-stage} produces a consistent estimator with $N = \omega( \log(d))$ samples.

\section{Results on the Noisy Input Model}\label{sec:noisy-input}
Recall the Noisy Input Model: $Y = f(X + W)$, where $f$ is an $s$-sparse coordinate-wise monotone function with active coordinates $A$. We assume throughout this section that $X$ is a uniform random variable on $[0,1]^d$, $W$ is a zero-mean random variable independent from $X$, and $f: \mathbb{R}^d \to \{0,1\}$. 

In this section, we prove the statistical consistency of Two-Stage Isotonic Regression, with Sequential Linear Programming Support Recovery as the support recovery algorithm. In Subsection \ref{subsec:noisy-input-consistency} we consider the setting where the set of active coordinates is known, and provide an upper bound on the resulting $L_2$-norm error of our estimator. In Subsection \ref{subsection:noisy-input-support-recovery} we provide a guarantee on the probability of correctly estimating the support, using S-LPSR. These results are combined to give Corollary \ref{corollary:two-stage-noisy-input}, stated at the end of the section. As a special case of the corollary, if $s$ is constant and $N = \omega(\log(d))$ samples are used by TSIR, then the estimator $\hat{f}_n$ that is produced is consistent.

\subsection{Statistical consistency}\label{subsec:noisy-input-consistency}
Suppose that the set of active coordinates, $A$, is known. Then we can apply Problem \eqref{eq:two-stage-first}-\eqref{eq:two-stage-last} within Algorithm \ref{alg:two-stage} to estimate the function values, with the variables $v_i$ that indicate the active coordinates set to $1$ if $i \in A$, and set to $0$ otherwise. The coordinates outside the active set do not influence the solution of the optimization problem, and therefore do not affect the estimated function. Therefore, the setting where $A$ is known is equivalent to the non-sparse setting with dimension $d = s$.

We investigate the regime under which Problem \eqref{eq:two-stage-first}-\eqref{eq:two-stage-last} produces a consistent estimator, in the non-sparse setting ($d=s$). To state our guarantees, it is convenient to represent binary coordinate-wise monotone functions in terms of \emph{monotone partitions}. 
\begin{definition}[Monotone Partition]
We say that $(S_0, S_1)$ is a \emph{monotone partition} of $\mathbb{R}^d$ if
\begin{enumerate}
\item $S_0$ and $S_1$ form a partition of $\mathbb{R}^d$. That is, $S_0 \cup S_1 = \mathbb{R}^d$ and $S_0 \cap S_1 = \emptyset$.
\item For all $x, y \in \mathbb{R}^d$, if $x \preceq y$, then either (i) $x,y \in S_0$, (ii) $x, y \in S_1$, or (iii) $x \in S_0, y \in S_1$.
\end{enumerate}
Let $\mathcal{M}_d$ be the set of all monotone partitions of $\mathbb{R}^d$.
\end{definition}
Note that there is a one-to-one correspondence between monotone partitions and binary coordinate-wise monotone functions.

Let $Y = f(X+W)$ represent our model, with $d = s$, and with $f$ corresponding to a monotone partition $(S_0^{\star}, S_1^{\star})$. That is, $f(x) = 0$ for $x \in S_0^{\star}$ and $f(x) = 1$ for $x \in S_1^{\star}$.  Let $h_0(x)$ be the probability density function of $X$, conditional on $Y = 0$. Similarly, let $h_1(x)$ be the probability density function of $X$, conditional on $Y = 1$. For $(S_0, S_1) \in \mathcal{M}_d$, let 
\begin{align*}
H_0(S_1) = \int_{z \in S_1} h_0(z) dz ~\text{ and } ~H_1(S_0) = \int_{z \in S_0} h_1(z) dz.
\end{align*}
Finally, let $p$ be the probability that $Y = 0$. Let 
\[q(S_0, S_1) \triangleq p H_0(S_1) + (1-p) H_1(S_0).\]
The value of $q(S_0, S_1)$ is the probability of misclassification, under the monotone partition $(S_0, S_1)$. 

\begin{assumption}\label{assumption:unique-minimizer}
We assume that $q$ has a unique minimizer on $\mathcal{M}_d$, which is $(S_0^{\star}, S_1^{\star})$. 
\end{assumption}

\begin{definition}[Discrepancy]
For two monotone partitions $(S_0, S_1)$ and $(S_0', S_1')$, the discrepancy function $D: \mathcal{M}_d \times \mathcal{M}_d \to [0,1]$ is defined as follows.
\[D\left((S_0, S_1), (S_0', S_1')\right) \triangleq \mathbb{P}\left(X \in S_0 \cap S_1' \right) + \mathbb{P}\left(X \in S_0' \cap S_1\right)  \]
Also let
\[B_{\delta}\left(S_0^{\star}, S_1^{\star}\right) \triangleq \{(S_0, S_1) \in \mathcal{M}_d :  D\left((S_0, S_1), (S_0^{\star}, S_1^{\star})\right)\leq \delta \}\] 
be the set of monotone partitions with discrepancy at most $\delta$ from $(S_0^{\star}, S_1^{\star})$.
\end{definition}



\begin{theorem}\label{thm:consistency-binary}
Let $d = s$. Suppose Assumption \ref{assumption:unique-minimizer} holds, and the components of $W$ are independent. Let $\hat{f}_n$ be the estimator derived from Algorithm \ref{alg:sequential-recovery}, and let 
\[q_{\text{min}}(\delta) \triangleq \min \left\{q(S_0, S_1): (S_0, S_1) \not \in B_{\delta}(S_0^{\star}, S_1^{\star})\right\} > q(S_0^{\star}, S_1^{\star}).\] 
Then for any $0 < \delta \leq 1$, 
\begin{align*}
&\mathbb{P} \left( \Vert \hat{f} - f \Vert_2 >\delta \right) \leq\\
&\frac{ \exp \left[ \left(2^s + 2\log(2) -1 \right) n^{\frac{s-1}{s}}\right] }{\exp \left[n^{\frac{2s-1}{2s} }\right]} +\left(\exp \left[n^{\frac{2s-1}{2s}}\right] + 1\right) \exp\left(- \frac{ \left(q_{\text{min}}\left(\delta\right) - q\left(S_0^{\star}, S_1^{\star} \right) \right)^2 n}{36} \right).
\end{align*}
\end{theorem}


\begin{corollary}\label{corollary:noisy-input}
Suppose that $q_{\text{min}}\left(\delta\right) - q\left(S_0^{\star}, S_1^{\star} \right) = \Theta(1)$, that is, constant in $s$.
When $d = s$ and $ n =  e^{\omega(s^2)}$, the estimator $\hat{f}_n$ produced by Algorithm \ref{alg:simultaneous} is consistent. 
\end{corollary}
Theorem \ref{thm:consistency-binary} has an analogous version in the sparse setting ($s < d$), which we give in the supplementary material (Theorem \ref{thm:noisy-input-sparse}). The result allows us to state the following corollary regarding the IPIR algorithm.
\begin{corollary}\label{corollary:noisy-input-simultaneous}
Suppose $s$ is constant and the components of $W$ are independent. Let $\hat{f}_n$ be the estimator produced by Algorithm \ref{alg:simultaneous}. If $n = \omega(\log(d))$, then $\hat{f}_n$ is a consistent estimator.
\end{corollary}

\subsection{Support recovery}\label{subsection:noisy-input-support-recovery}
In this subsection, we give support recovery guarantees for Algorithm \ref{alg:sequential-recovery}. The guarantees will be in terms of differences of probabilities.
\begin{definition}
Let $Y_1 = f(X_1 + W_1)$ and $Y_2 = f(X_2 + W_2)$ be two independent samples from the model. For $k \in A$, define
\[\overline{p}_{k} \triangleq \mathbb{P} \left( Y_1 = 1, Y_2 = 0~|~q(1,2,k) = 1 \right) - \mathbb{P} \left( Y_1 = 0, Y_2 = 1~|~q(1,2,k) = 1 \right).\]
Assume without loss of generality that $A = \{1, \dots, s\}$ and $\overline{p}_1 \leq \overline{p}_2 \leq \dots \leq \overline{p}_s$.
\end{definition}
\begin{theorem}\label{thm:support-recovery-noisy-input-multiple}
Let $B$ be the set of indices corresponding to running Algorithm \ref{alg:sequential-recovery}. 
It holds that $B = A$ with probability at least \[1 - (d-s) \sum_{k=1}^s (d+1 -k)  \exp\left(-\frac{n \overline{p}_{k}^2}{16} \right).\]
\end{theorem}
We can now give a guarantee of the success of Algorithm \ref{alg:two-stage}, using Algorithm \ref{alg:sequential-recovery} for support recovery.


\begin{corollary}\label{corollary:noisy-input-support-recovery}
Suppose that $\overline{p}_1 = \Theta(1)$. Let $N = s \cdot n$ be the number of samples used by Algorithm \ref{alg:sequential-recovery}. If $N =s \cdot \omega( \log(d))$, then Algorithm \ref{alg:sequential-recovery} recovers the true support w.h.p. as $s, d \to \infty$.
\end{corollary}

\begin{corollary}\label{corollary:two-stage-noisy-input}
Suppose that $q_{\text{min}}\left(\delta\right) - q\left(S_0^{\star}, S_1^{\star} \right) = \Theta(1)$.
Suppose also that $\overline{p}_1 = \Theta(1)$, and that the components of $W$ are independent.
Consider running Algorithm \ref{alg:two-stage} using $s \times n$ samples for sequential support recovery and an additional $n$ samples for function value estimation. Let $N = (s+1) n$ be the total number of samples, and let $\hat{f}_N$ be the estimated function. If $N = (s+1)\omega( \log(d))$ and $N = (s+1)e^{\omega(s^2)}$, then $\hat{f}_N$ is a consistent estimator.
\end{corollary}

\section{Experimental results}\label{sec:experiments}
All algorithms were implemented in Java version 8, using Gurobi version 6.0.0.
\subsection{Support recovery}
We test the support recovery algorithms on random synthetic instances. 
Let $A = \{1, \dots, s\}$ without loss of generality. First, randomly sample $r$ ``anchor points'' in $[0,1]^d$, calling them $Z_1, \dots, Z_r$. The parameter $r$ governs the complexity of the function produced. In our experiment, we set $r = 10$. Next, randomly sample $X_1, \dots, X_n$ in $[0,1]^d$. For $i \in \{1, \dots, n\}$, assign $Y_i = 1 + W_i$ if $Z_j\preceq_{A} X_i$ for some $j \in \{1, \dots, r\}$, and assign $Y_i = W_i$ otherwise. The linear programming based algorithms for support recovery, LPSR and S-LPSR, are compared to the simultaneous approach, IPIR, which estimates the active coordinates while also estimating the function values. Note that even though the proof of support recovery using S-LPSR requires fresh data at each iteration, our experiments do not use fresh data. We keep $s = 3$ fixed and vary $d$ and $n$. The error is Gaussian with mean $0$ and variance $0.1$, independent across coordinates. We report the percentages of successful recovery (see Table \ref{table:synthetic}). The IPIR algorithm performs the best on nearly all settings of $(n,d)$. This suggests that the objective of the IPIR algorithm- to minimize the number of misclassifications on the data- gives the algorithm an advantage in selecting the true active coordinates. The S-LPSR algorithm generally does better than the LPSR algorithm; for $n=250$ samples, they perform about the same when $d \in \{5, 10\}$ but for $d=20$, the LPSR algorithm succeeds $45\%$ of the time while the S-LPSR algorithm succeeds $70\%$ of the time, and when $d=50$, the LPSR algorithm was not able to recover the correct coordinates on any trial, while the S-LPSR algorithm recovered them $55\%$ of the time. It appears that determining the coordinates one at a time provides implicit regularization.

\begin{table}[htp]
  \caption{Performance of support recovery algorithms on synthetic instances. Each line of the table corresponds to 20 trials.}
  \label{table:synthetic}
  \centering
  \begin{tabular}{ccccccccccccc}
  \toprule
  \multirow{3}{*}{} & \multicolumn{4}{c}{\textbf{IPIR}} & \multicolumn{4}{c}{\textbf{LPSR}} &  \multicolumn{4}{c}{\textbf{S-LPSR}} \\
    & \multicolumn{4}{c}{$d=$} & \multicolumn{4}{c}{$d=$} & \multicolumn{4}{c}{$d=$}\\
    $n$ & 5 & 10 &  20 & 50 & 5 & 10 &  20 & 50  & 5 &  10 & 20 & 50 \\
    \midrule
    $50$ &65 & 60 & 60 & 40 & 70 & 25 & 5 & 0 & 55 & 50& 15 & 5\\
    $100$ &90 & 90 & 70 & 70 & 95 & 40 & 20 & 0 & 65 & 65 & 55 & 20\\
     $150$ &100 & 100 & 95 & 90 & 100 & 60 & 30 & 0 & 95 & 80 & 50 & 45\\
     $200$ & 100 & 100 & 90 & 95 & 100 & 50 & 35 & 5 & 100 & 90 & 65 & 40\\
     $250$ & 100 & 100 & 90 & 90 & 95 & 75 & 45 & 0 & 90 & 75 & 70 & 55\\
    \bottomrule
  \end{tabular}
\end{table}


\subsection{Cancer classification using gene expression data}
In order to assess the applicability of our sparse monotone regression approach, we apply it to cancer classification using gene expression data. The data is drawn from the COSMIC database \cite{Forbes2010}, which is widely used in quantitative research in cancer biology. Each patient in the database is identified as having a certain type of cancer. For each patient, gene expressions are reported as a z-score. Namely, if $\mu_G$ and $\sigma_G$ are the mean and standard deviation of the gene expression of gene $G$ and $x$ is the gene expression of a certain patient, then his or her z-score would be equal to $\frac{x - \mu_G}{\sigma_G}$. 
We filter the patients by cancer type, selecting those with skin and lung cancer, two common cancer types. There are $236698$ people with lung or skin cancer in the database, though the database only includes gene expression data for $1492$ of these individuals. Of these, $1019$ have lung cancer and $473$ have skin cancer. A classifier always selecting ``lung'' would have an expected correct classification rate of $1019/1492 \approx 68\%$. Therefore this rate should be regarded as the baseline classification rate.

Our goal is to use gene expression data to classify the patients as having either skin or lung cancer. We associate skin cancer as a ``0'' label and lung cancer as a ``1'' label.
We only include the $20$ most associated genes for each of the two types, according to the COSMIC website. This leaves $31$ genes, since some genes appear on both lists. We additionally include the negations of the gene expression values as coordinates, since a lower gene expression of certain genes may promote lung cancer over skin cancer. The number of coordinates is therefore equal to $62$. The number of active genes is ranged between $1$ and $5$.


We perform both simultaneous and two-stage isotonic regression, comparing the IPIR and TSIR algorithms, using S-LPSR to recover the coordinates in the two-stage approach. Since for every gene, its negation also corresponds to a coordinate, we added additional constraints. In IPIR, we use variables $v_k \in \{0,1\}$ to indicate whether coordinate $k$ is in the estimated set of active coordinates. In LPSR and S-LPSR, we use variables $v_k \in [0,1]$ instead. In order to incorporate the constraints regarding negation of coordinates in IPIR, we included the constraint $v_i + v_j \leq 1$ for pairs $(i,j)$ such that coordinate $j$ is the negation of coordinate $i$. In S-LPSR, once a coordinate $v_i$ was selected, its negation was set to zero in future iterations. The LPSR algorithm, however, could not be modified to take this additional structure into account without using integer variables. Adding the constraints $v_i + v_j \leq 1$ when coordinate $j$ is the negation of coordinate $i$ proved to be insufficient. Therefore, we do not include the LPSR algorithm in our experiments on the COSMIC database.

We compare our isotonic regression algorithms to two classical algorithms: $k$-Nearest Neighbors (\cite{Fix1951}) and the Support Vector Machine (\cite{Cortes1995}). Given a test sample $x$ and an odd number $k$, the $k$-Nearest Neighbors algorithm finds the $k$ closest training samples to $x$. The label of $x$ is chosen according to the majority of the labels of the $k$ closest training samples. The SVM algorithm used is the soft-margin classifier with penalty $C$ and polynomial kernel given by $K(x,y) = (1 + x \cdot y)^m$.

In Table \ref{table:main-results}, each row is based on 10 trials, with 1000 test data points chosen uniformly and separately from the training points. 
The two-stage method was generally faster than the simultaneous method. With $200$ training points and $s=3$, the simultaneous method took $260$ seconds on average per trial, while the two-stage method took only $42$ seconds per trial. The simultaneous method became prohibitively slow for higher values of $n$. The averages for $k$-Nearest Neighbors and Support Vector Machine are taken as the best over parameter choices in hindsight. For $k$-Nearest Neighbors, $k \in \{1,3,5,7,9,11,15\}$, and for SVM, $C \in \{10,100,500,1000 \}$ and $m \in \{1,2,3,4\}$. The fact that the sparse isotonic regression method outperforms the $k$-NN classifier and the polynomial kernel SVM by such a large margin can be explained by a difference in structural assumptions; the results suggest that monotonicity, rather than proximity or a polynomial functional relationship, is the correct property to leverage.

\begin{table}[]
\caption{Comparison of classifier success rates on COSMIC data. Top row data is according to the ``min'' interpolation rule and bottom row data is according to the ``max'' interpolation rule.}\label{table:main-results}
\begin{tabular}{ccccccccccccc}
\toprule
\multirow{3}{*}{$n$} & \multicolumn{5}{c}{\textbf{IPIR}} & \multicolumn{5}{c}{\textbf{TSIR + S-LPSR}} & \textbf{$k$-NN} & \textbf{SVM} \\
 & \multicolumn{5}{c}{$s=$}  & \multicolumn{5}{c}{$s=$} & & \\
  & 1 & 2 & 3 & 4 & 5  & 1 & 2 & 3 & 4 & 5  & &\\
 \midrule
 \multirow{2}{*}{100}  &83.1 & 84.6& 76.8&66.2 &53.8  & 82.4&84.6 &77.8 &73.0 &65.4 &69.8 &63.8\\
  & 83.9&91.8 &91.0 & 85.7& 75.7& 82.9 &90.4 &88.9 &87.4 &83.3 & &\\
  \midrule
 \multirow{2}{*}{200}  & 85.4&88.1 &84.3 &73.9 &62.7  &85.4 &89.3 &86.7 &81.2 &76.9 &76.6 &72.6\\
  &85.8 &92.6 &96.4 &88.9 &83.9  &85.8 &94.5 &95.9 &95.3 &93.0 & &\\
  \midrule
 \multirow{2}{*}{300}  & -& -& -& -& - &84.7 &91.7 &89.0 &84.4 &80.2 &76.6 &74.2\\
  &- & -&- & -&  -& 85.1& 94.2& 95.6&95.9 &94.8 & &\\
  \midrule
 \multirow{2}{*}{400} & -&- &- & -& -&85.6 &91.8 &89.7 &87.3 &81.7 &78.6 &77.4\\
 & -&- & -& -& -&85.8 &94.0 &95.7 &96.4&95.7 & &\\
 \bottomrule
\end{tabular}
\end{table}



The results suggest that the correct sparsity level is $s = 3$. With $n = 400$ samples, the classification accuracy rate is $95.7\%$. When the sparsity level is too low, the monotonicity model is too simple to accurately describe the monotonicity pattern. On the other hand, when the sparsity level is too high, fewer points are comparable, which leads to fewer monotonicity constraints.
For $n \in \{100, 200\}$ and $d \in \{1,2,3,4,5\}$, TSIR + S-LPSR does at least as well as IPIR on 15 out of 20 of $(n,d)$ pairs, and outperforms on 12 of these. This result is surprising, because synthetic experiments show that IPIR outperforms S-LPSR on support recovery. 

We further investigate the TSIR + S-LPSR algorithm. Figure \ref{fig:labellings} shows how the two-stage procedure labels the training points. 
The high success rate of the sparse isotonic regression method suggests that this nonlinear picture is quite close to reality. The observed clustering of points may be a feature of the distribution of patients, or could be due to a saturation in measurement.
Figure \ref{fig:robustness} studies the robustness of TSIR + S-LPSR. Additional synthetic zero-mean Gaussian noise is added to the inputs, with varying standard deviation. The ``max'' classification rule is used. $200$ training points and $1000$ test points were used. Ten trials were run, with one standard deviation error bars indicated in gray. The results indicate that TSIR + S-LPSR is robust to moderate levels of noise. 
\begin{figure}[h]
\centering
\begin{subfigure}{0.45\textwidth}
\centering
\includegraphics[trim={2.5cm 8cm 2cm 7.5cm},clip,width=\textwidth]{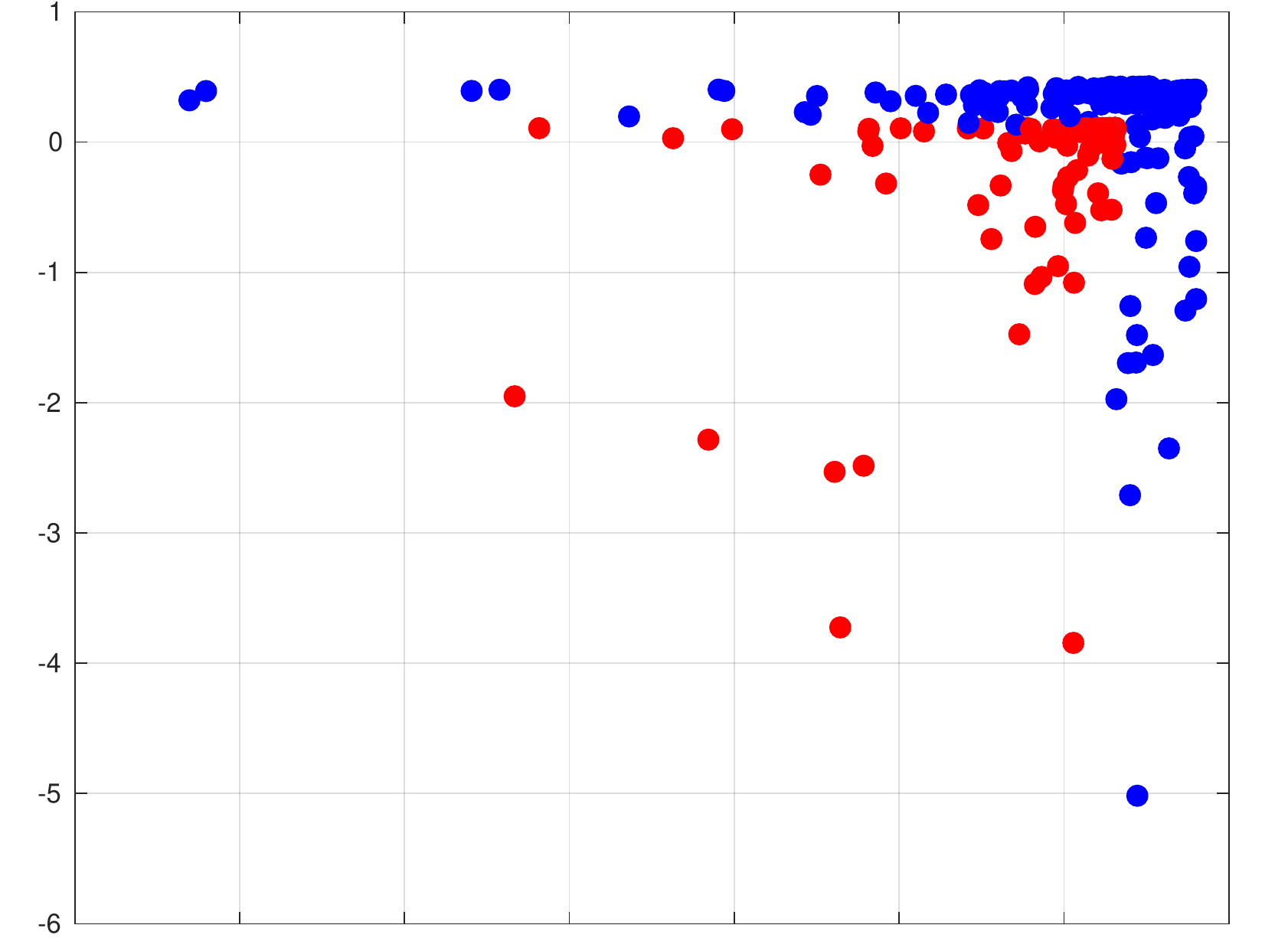}
\caption{$s=2$.} 
\label{fig:labellings-s=2}
\end{subfigure}
\begin{subfigure}{0.45\textwidth}
\centering
\includegraphics[trim={2.5cm 8cm 2cm 7.5cm},clip,width=\textwidth]{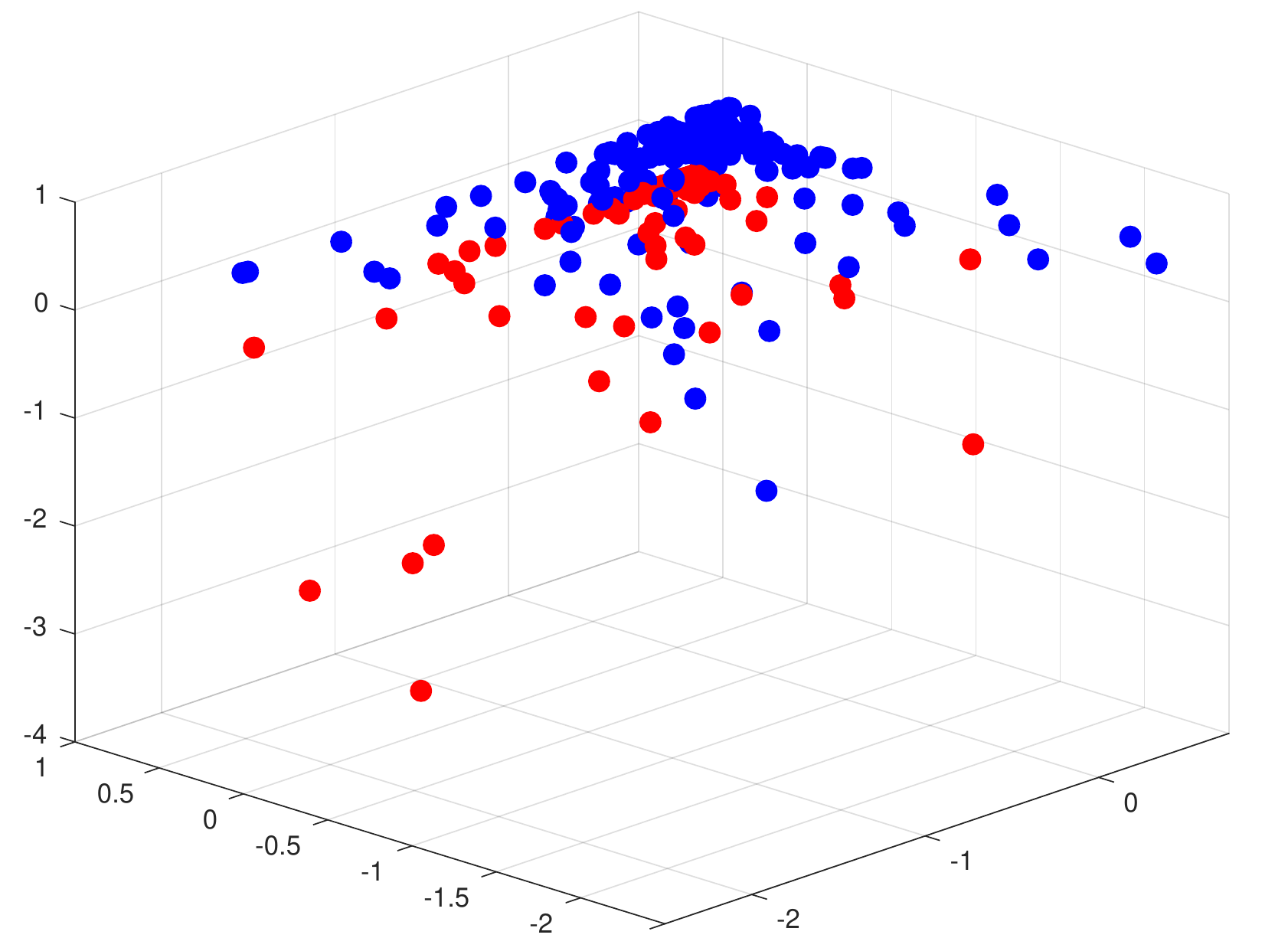}
\caption{$s=3$.} 
\label{fig:labellings-s=3}
\end{subfigure}
\caption{Illustration of the TSIR + S-LPSR algorithm. Blue and red markers correspond to lung and skin cancer, respectively.}
\label{fig:labellings}
\end{figure}


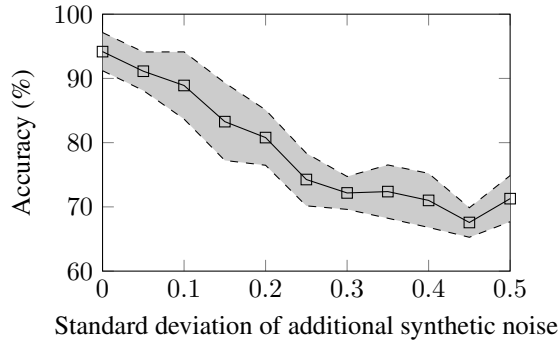
\begin{figure}[h]
\centering
\begin{tikzpicture}[scale=1]
\begin{axis}[
    width=7cm,height=5cm,
    xlabel={Standard deviation of additional synthetic noise},
    ylabel={Accuracy (\%)},
    xmin=0, xmax=0.5,
    ymin=60, ymax=100,
    xtick={0, 0.1, 0.2, 0.3, 0.4, 0.5},
    ytick={50,60,70,80,90,100},
    grid style=dashed,
]
 
\addplot[name path=mean,
    color=black,
    mark=square,
    ]
    coordinates {
    (0,94.17)(0.05,91.12)(0.1,88.9)(0.15,83.27)(0.2,80.79)(0.25,74.26)(0.3,72.17)(0.35,72.38)(0.4,71.01)(0.45,67.57)(0.5,71.28)
    };

\addplot[name path=mean-plus,
    color=black,
    dashed
    ]
    coordinates {
    (0,97.14)(0.05,94.12)(0.1,94.13)(0.15,89.31)(0.2,85.10)(0.25,78.35)(0.3,74.73)(0.35,76.53)(0.4,75.22)(0.45,69.87)(0.5,74.87)
    }; 
    
    \addplot[name path=mean-minus,
    color=black,
    dashed
    ]
    coordinates {
    (0,91.20)(0.05,88.12)(0.1,83.67)(0.15,77.23)(0.2,76.48)(0.25,70.16)(0.3,69.61)(0.35,68.23)(0.4,66.80)(0.45,65.27)(0.5,67.69)
    }; 
    
    \addplot[gray!40] fill between[of=mean and mean-plus];   
    \addplot[gray!40] fill between[of=mean and mean-minus];    
\end{axis}
\end{tikzpicture}
\caption{Robustness to error.}\label{fig:robustness}
\end{figure}

\FloatBarrier
\section{Conclusion}
In this paper, we have considered the sparse isotonic regression problem under two noise models: Noisy Output and Noisy Input. We have formulated optimization problems to recover the active coordinates, and then estimate the underlying monotone function. We provide explicit guarantees on the performance of these estimators. Finally, we demonstrate the applicability of our approach to a cancer classification task, showing that our methods outperform widely-used classifiers. While the task of classifying patients with two cancer types is relatively simple, the accuracy rates illustrate the modeling power of the sparse monotone regression approach. 
\clearpage
\bibliographystyle{plain}
\bibliography{MIT_Project_4}


\FloatBarrier
\clearpage
\setcounter{section}{0}
\section{Algorithms}
This section gives further detail on the algorithms introduced in the paper.

Problem \eqref{eq:optimization-first}-\eqref{eq:optimization-last} can be encoded as a single mixed-integer convex minimization problem, as follows. The algorithm is below.
Binary variables $v_k$ indicate the estimated active coordinates; $v_k = 1$ means that the optimization program has determined that coordinate $k$ is active. The variables $F_i$ represent the estimated function values at data points $X_i$.
\begin{breakablealgorithm}
\caption{Integer Programming Isotonic Regression (IPIR)}\label{alg:simultaneous}
\begin{algorithmic}[1]
\Require{Values $(X_1, Y_1), \dots, (X_n, Y_n)$; sparsity level $s$}
\Ensure{An estimated function $\hat{f}_n$}
\State Solve the following optimization problem.
\begin{align}
&\min_{v, F} \sum_{i=1}^n \left(Y_i - F_i\right)^2 \label{eq:simultaneous-first}\\
\text{s.t. } &\sum_{k=1}^d v_k = s \\
& \sum_{k=1}^d q(i,j,k) v_k \geq F_i - F_j & \forall i, j \in \{1, \dots, n\} \label{eq:constraint-simultaneous-monotonicity}\\
&v_k \in \{0,1\} &\forall k \in \{1, \dots, d\}\\
&F_i \in \mathcal{R} &\forall i \in \{1, \dots, n\}  \label{eq:simultaneous-last}
\end{align}
\State Return the function $\hat{f}_n(x) = \min\{y \in \mathcal{R} : X_i \preceq x \implies y \geq F_i\}$.
\end{algorithmic}
\end{breakablealgorithm}
We claim that Problem \eqref{eq:simultaneous-first}-\eqref{eq:simultaneous-last} is equivalent to Problem \eqref{eq:optimization-first}-\eqref{eq:optimization-last}. Indeed, the monotonicity requirement is $X_i \preceq_{A} X_j \implies f(X_i) \leq f(X_j)$. The contrapositive of this statement is $f(X_i) > f(X_j) \implies X_i \not \preceq_{A} X_j$; alternatively, $f(X_i) > f(X_j) \implies \exists k \in A \text{ s.t. } X_{ik} > X_{jk}$. The contrapositive is expressed by Constraints \eqref{eq:constraint-simultaneous-monotonicity}. 

Recall that in the Noisy Input Model, the function $f$ is binary-valued, i.e. $\mathcal{R} = \{0, 1\}$. Let $\mathcal{S}^+ = \{i : Y_i = 1\}$ and $\mathcal{S}^- = \{i : Y_i = 0\}$. Note that if we replace the objective function \eqref{eq:simultaneous-first} by $\sum_{i \in \mathcal{S}^+} \left(1 - F_i \right) + \sum_{i \in \mathcal{S}^-} F_i$, then we obtain an equivalent formulation, which is a linear integer program.

We now give details for the two methods of support recovery, which are used in the two-stage approach.
\begin{breakablealgorithm}
\caption{Linear Programming Support Recovery (LPSR)}\label{alg:simultaneous-recovery}
\begin{algorithmic}[1]
\Require{Values $(X_1, Y_1), \dots, (X_n, Y_n)$; sparsity level $s$}
\Ensure{The estimated support, $\hat{A}$}
\State Solve the following optimization problem.
\begin{align}
& \min_{v,c} \sum_{i=1}^n \sum_{j=1}^n \sum_{k=1}^d c^{ij}_k \label{eq:support-recovery-first}\\
\text{s.t. } &\sum_{k=1}^d v_k = s\\
&\sum_{k=1}^d q(i,j,k) \left(v_k + c^{ij}_k\right) \geq 1 & \text{if } Y_i > Y_j \text{ and } \sum_{k=1}^d q(i,j,k) \geq 1 \label{eq:support-recovery-monotonicity}\\
&0 \leq v_k \leq 1 & \forall k \in \{1, \dots, d\}\\
&c^{ij}_k \geq 0 & \forall i \in \{1, \dots, n\}, j \in \{1, \dots, n\}, k \in \{1, \dots, p\} \label{eq:support-recovery-last}
\end{align}
\State Determine the $s$ largest values $v_i$, breaking ties arbitrarily. Let $\hat{A}$ be the set of the corresponding $s$ indices.
\end{algorithmic}
\end{breakablealgorithm}
In Problem \eqref{eq:support-recovery-first}-\eqref{eq:support-recovery-last}, the $v_k$ variables are meant to indicate the active coordinates, while the $c^{ij}_k$ variables act as correction in the monotonicity constraints. For example, if for one of the constraints \eqref{eq:support-recovery-monotonicity}, $\sum_{k=1}^d q(i,j,k) v_k = 0.7$, then we will need to set $c_k^{ij}= 0.3$ for some $(i,j,k)$ such that $q(i,j,k) = 1$. The $v_k$'s should therefore be chosen in a way to minimize the correction. 

Algorithm \ref{alg:sequential-recovery} determines the active coordinates one at a time, setting $s = 1$ in Problem \eqref{eq:support-recovery-first}-\eqref{eq:support-recovery-last}. Once a coordinate $i$ is included in the set of active coordinates, variable $v_i$ is set to zero in future iterations.
\begin{breakablealgorithm}
\caption{Sequential Linear Programming Support Recovery (S-LPSR)}\label{alg:sequential-recovery}
\begin{algorithmic}[1]
\Require{Values $(X_1, Y_1), \dots, (X_n, Y_n)$; sparsity level $s$}
\Ensure{The estimated support, $\hat{A}$}
\State $B \gets \emptyset$
\While {$|B| < s$} 
\State Solve the optimization problem in Algorithm \ref{alg:simultaneous-recovery} with $s = 1$:
\begin{align}
& \min \sum_{i=1}^n \sum_{j=1}^n \sum_{k=1}^d c^{ij}_k \label{eq:support-recovery-sequential-first}\\
\text{s.t. } &\sum_{k=1}^d v_k = 1\\
&v_i = 0 & \forall i \in B\\
&\sum_{k=1}^d q(i,j,k) \left(v_k + c^{ij}_k\right) \geq 1 & \text{if } Y_i > Y_j \text{ and } \sum_{k=1}^d q(i,j,k) \geq 1\\
&0 \leq v_k \leq 1 & \forall k \in \{1, \dots, d\}\\
&c^{ij}_k \geq 0 & \forall i \in \{1, \dots, n\}, j \in \{1, \dots, n\}, k \in \{1, \dots, d\} \label{eq:support-recovery-sequential-last}
\end{align}
\State Identify $i_{\text{max}}$ such that $v_{\text{max}} = \max\{v_i\}$, breaking ties arbitrarily. Set $B \gets B \cup \{i_{\text{max}}\}$.
\EndWhile
\State Return $\hat{A} = B$.
\end{algorithmic}
\end{breakablealgorithm}

We are now ready to state the two-stage algorithm for estimating the function $\hat{f}_n$.
\begin{breakablealgorithm}
\caption{Two Stage Isotonic Regression (TSIR)}\label{alg:two-stage}
\begin{algorithmic}[1]
\Require{Values $(X_1, Y_1), \dots, (X_n, Y_n)$; sparsity level $s$}
\Ensure{The estimated function, $\hat{f}_n$}
\State Estimate $\hat{A}$ by using Algorithm \ref{alg:simultaneous-recovery} or \ref{alg:sequential-recovery}. Let $v_k = 1$ if $k \in \hat{A}$ and $v_k = 0$ otherwise.
\State Solve the following optimization problem.
\begin{align}
&\min \sum_{i=1}^n \left(Y_i - F_i\right)^2 \label{eq:two-stage-first}\\
\text{s.t. } & \sum_{k=1}^d q(i,j,k) v_k \geq F_i - F_j & \forall i, j \in \{1, \dots, n\} \label{eq:two-stage-monotonicity}\\
&F_i \in \mathcal{R} &\forall i \in \{1, \dots, n\}  \label{eq:two-stage-last}
\end{align}
In the Noisy Input Model, replace the objective with $\sum_{i \in \mathcal{S}^+} \left(1 - F_i \right) + \sum_{i \in \mathcal{S}^-} F_i$.
\State Return the function $\hat{f}_n(x) = \min\{y \in \{0,1\} : X_i \preceq x \implies y \geq F_i\}$.
\end{algorithmic}
\end{breakablealgorithm}
\begin{lemma}\label{lemma:integer-solution}
Under the Noisy Input Model, replacing the constraints $F_i \in \{0,1\}$ with $F_i \in [0,1]$ in Problems \eqref{eq:simultaneous-first}-\eqref{eq:simultaneous-last} and \eqref{eq:two-stage-first}-\eqref{eq:two-stage-last} does not change the optimal value. Furthermore, there always exists an integer optimal solution.
\end{lemma}

\begin{proof}
Consider Problem \eqref{eq:two-stage-first}-\eqref{eq:two-stage-last}, with the objective function replaced by $\sum_{i \in \mathcal{S}^+} \left(1 - F_i \right) + \sum_{i \in \mathcal{S}^-} F_i$. Here, the vector $v$ is fixed. Since $F_i \in [0,1]$ for all $i \in \{1, \dots, n\}$, $F_i - F_j \in [-1,1]$. The left side of Constraint \eqref{eq:two-stage-monotonicity} takes value in $\{0, 1, \dots, d\}$. Therefore, the constraint is tight only when the left side is equal to $0$. Therefore, we only require that for $(i,j)$ such that $\sum_{k=1}^d q_{i,j,k} v_k = 0$, it holds that $F_i \leq F_j$. Let $\epsilon = \min \left \{ \min_{F_i > 0} \{F_i\},  \min_{F_i < 1} \{1- F_i\}\right\}$. In other words, $\epsilon$ is the margin to the endpoints $[0,1]$. Suppose that $F$ is an optimal solution with some values $F_i \in (0,1)$. Then $\epsilon > 0$. Let $C = \{i : F_i \in (0,1)$. Consider adding $\epsilon$ to each $F_i$ such that $i \in C$, and call the new solution $F^{+ \epsilon}$. Clearly, $F^{+ \epsilon}$ is feasible. The change in the objective is equal to 
\begin{align*}
&\sum_{i \in \mathcal{S}^+} \left(1 - F^{+\epsilon}_i \right) + \sum_{i \in \mathcal{S}^-} F^{+\epsilon}_i  -\sum_{i \in \mathcal{S}^+} \left(1 - F_i \right) - \sum_{i \in \mathcal{S}^-} F_i  \\
&= \epsilon \left( \left|\left\{i : i \in C, i \in \mathcal{S}^-\right\} \right| -  \left|\left\{i : i \in C, i \in \mathcal{S}^+\right\} \right| \right)
\end{align*}
On the other hand, consider subtracting $\epsilon$ from each $F_i$ such that $i \in C$, and call the new solution $F^{- \epsilon}$. By construction, $F^{- \epsilon}$ is also feasible. The chance in the objective is equal to $\epsilon \left( \left|\left\{i : i \in C, i \in \mathcal{S}^+\right\} \right| -  \left|\left\{i : i \in C, i \in \mathcal{S}^-\right\} \right| \right)$. Since we have assumed that $F$ is an optimal solution, both changes must be nonnegative. Since they are negations of each other, they must both be equal to zero. Therefore, the solutions $F^{+ \epsilon}$ and $F^{- \epsilon}$ have the same objective value as the solution $F$. If $\epsilon = \min_{F_i > 0} \{F_i\}$, choose $F^{-\epsilon}$, and if $\epsilon = \min_{F_i < 1} \{1- F_i\}$, choose $F^{+\epsilon}$. This leads to the size of the set $C$ decreasing by one (or two). Repeating this process inductively, we eventually produce a solution with $C = \emptyset$. Therefore, we have shown that there always exists an integer optimal solution.

We have shown that for fixed $v$, there always exists an integer optimal solution. Varying $v$, it also holds that the optimal solutions to Problem \eqref{eq:simultaneous-first}-\eqref{eq:simultaneous-last} remains the same. 
\end{proof}

\clearpage
\section{Proofs for the Noisy Output Model}

We will build toward a proof of Theorem \ref{thm:statistical-consistency}. We note that Algorithm \ref{alg:simultaneous} selects an $s$-sparse coordinate-wise monotone function that minimizes the empirical $L_2$ loss. To prove the statistical consistency of the estimated function, we need to introduce the expected VC dimension (\cite{Vapnik1996}). 
Let $\mathcal{F}_{s,d}$ be the set of $s$-sparse coordinate-wise monotone functions on $[0,1]^d$. Following \cite{Gamarnik1999}, let $Q(x,y,f) = \left(y - f(x)\right)^2$ for $x \in [0,1]^d$, $y \in \mathbb{R}$, and $f \in \mathcal{F}_{s,d}$. 
For a fixed sequence $(x_1, y_1), \dots, (x_n, y_n) \in [0,1]^d \times [0,1]$, consider the set of vectors $\mathbf{Q} = \{\left(Q(x_1, y_1, f), \dots, Q(x_n, y_n, f)\right), f \in \mathcal{F}_{s,d}\}$. In other words, we vary over $\mathcal{F}_{s,d}$ and produce the associated error vectors. Let $N\left(\epsilon, \mathcal{F}_{s,d}, (x_1, y_1), \dots, (x_n, y_n) \right)$ be the size of the minimal $\epsilon$-net of the set $\mathbf{Q}$. Namely, a set $E$ is an $\epsilon$-net for $\mathbf{Q}$ if for every $q \in \mathbf{Q}$ there exists $v \in E$ such that $\Vert q - v \Vert_{\infty} \leq \epsilon$. For any $\epsilon > 0$, the \emph{expected VC entropy} of $\mathcal{F}_s$ is defined as
$$N_{\mathcal{F}_{s,d}} (\epsilon, n) = \mathbb{E} \left[N\left(\epsilon, \mathcal{F}_{s,d}, (X_1, Y_1), \dots, (X_n, Y_n) \right) \right].$$ The expectation is over the random variables $(X_i, Y_i)$. The expected VC entropy measures the complexity of the class $\mathcal{F}_{s,d}$,
and can be used to prove convergence in $L_2$.
\begin{proposition}[From Proposition 2 in \cite{Gamarnik1999}]\label{prop:statistical-consistency}
If $Y = f(X) + W \in [0,1]$ almost surely, then
\[\mathbb{P}\left( \Vert \hat{f}_n - f \Vert_2 > \epsilon \right) \leq 6 N_{\mathcal{F}_{s,d}} \left( \frac{\epsilon^2}{2^{10}}, n \right) e^{-\frac{3 \epsilon^3 n}{41 \times 2^{10}}}.\]
\end{proposition}
Therefore, if the expected VC entropy of $\mathcal{F}_{s,d}$ grows subexponentially in $n$, the estimator $\hat{f}_n$ derived from Algorithm \ref{alg:simultaneous} converges to the true function in $L_2$. Define the non-sparse class $\mathcal{F}_d = \mathcal{F}_{d,d}$. 
\begin{proposition}
$N_{\mathcal{F}_{s,d}}(\epsilon, n) \leq \binom{d}{s} N_{\mathcal{F}_{s}}(\epsilon, n)$.
\end{proposition}
\begin{proof}
The set $\mathcal{F}_{s,d}$ can be written as a union of $\binom{d}{s}$ function classes, depending on which subset of the coordinates is active. 
\end{proof}
Our goal is now to bound the expected VC entropy of the class $\mathcal{F}_d$. The expected VC entropy is related to a combinatorial quantity known as the \emph{labeling number}.
\begin{definition}[Labeling Number (\cite{Gamarnik1999})]
For a sequence of points $x_1, \dots, x_n \in [0,1]^d$ and a positive integer $m$, the labeling number $L(m, x_1, \dots, x_n)$ is the number of functions $\phi: \{x_1, \dots, x_n \} \to \{1, 2, \dots, m\}$ such that $\phi(X_i) \leq \phi(X_j)$ whenever $x_i \preceq x_j$, for $i, j \in \{1, \dots, n\}$. 
\end{definition}

\begin{proposition}\label{prop:labeling-number}
For any $(x_1, y_1), \dots, (x_n, y_n) \in [0,1]^d \times [0,1]$,
\[N\left(\epsilon, \mathcal{F}_d, (x_1, y_1), \dots, (x_n, y_n) \right) \leq L\left(\left \lceil \frac{2}{\epsilon} \right \rceil , x_1, \dots, x_n \right).\]
Let $\overline{\mathcal{F}}_d$ be the set of coordinate-wise monotone functions $f : [0,1]^d \to \{0, 1\}$. Then
\[N\left(\epsilon, \overline{\mathcal{F}}_d, (x_1, y_1), \dots, (x_n, y_n) \right) \geq L\left(\left \lfloor \sqrt{\frac{3}{2\epsilon}} \right \rfloor - 3 , x_1, \dots, x_n \right).\]
\end{proposition}
\begin{proof}
For the lower bound, let $\delta = \sqrt{\frac{2\epsilon}{3}}$, and let $N = \left \lfloor \frac{1}{\delta} \right \rfloor - 3$. Define the sequence $q_i = \delta (i + 1)$, for $i \in \{1, \dots, N\}$. The monotone labelings supported on $\{q_1, \dots, q_N\}$ are a subset of the coordinate-wise monotone functions. Our goal is to show that for every two distinct labelings $l_1$ and $l_2$, it holds that 
\[\Vert \left(Q(x_1, y_1, l_1), \dots, Q(x_n, y_n, l_1) \right), \left(Q(x_1, y_1, l_2), \dots, Q(x_n, y_n, l_2) \right)\Vert_{\infty} > 2 \epsilon.\]
If this relation holds for all distinct pairs of labelings, then at least $L(N, x_1, x_2, \dots, x_n)$ points are required to form an $\epsilon$-net of the set $\mathbf{Q}$. 

If $l_1$ and $l_2$ are distinct labelings, then there exists $k \in \{1, \dots, n\}$ such that $l_1(x_k) \neq l_2(x_k)$. Therefore,
\begin{align*}
\left| Q(x_k, y_k, l_1) - Q(x_k, y_k, l_2)\right| &= \left| (l_1(x_k) - y_k)^2 - (l_2(x_k) - y_k)^2 \right|\\
&= \left| l_1(x_k)^2 - 2 y_k l_1(x_k) + 2y_k l_2(x_k) - l_2(x_k)^2 \right|\\
&= \left| 2y_k \left( l_1(x_k) - l_2(x_k)\right) + \left( l_1(x_k) - l_2(x_k)\right) \left(l_1(x_k) + l_2(x_k)\right) \right|\\
&= \left| l_1(x_k) - l_2(x_k) \right| \left| l_1(x_k) + l_2(x_k) - 2y_k \right|\\
&\geq \delta \cdot 2  \left| \frac{l_1(x_k) + l_2(x_k)}{2} - y_k \right|\\
&\geq \delta \cdot 2 \min \left\{q_1, 1 - q_N \right\}\\
&\geq 4\delta^2\\
&= \frac{8}{3}\epsilon\\
&> 2 \epsilon
\end{align*}
We conclude that $L(N, \overline{\mathcal{F}}_d, x_1, \dots, x_n) \leq N\left(\epsilon, \overline{\mathcal{F}}_d, (x_1, y_1), \dots, (x_n, y_n) \right)$.

For the upper bound, the proof comes from the proof of Proposition 3 in \cite{Gamarnik1999}. Let $N = \left \lceil \frac{2}{\epsilon} \right \rceil$. Let $q_i = \frac{i-1}{N}$ for $i \in \{1, \dots, N, N+1\}$. Define
\[G \triangleq \{\left((y_1 - g_1)^2, (y_2 - g_2)^2, \dots, (y_n - g_n)^2 \right) : g_i \in \{q_1, \dots, q_N\}, x_i \preceq x_j \implies g_i \leq g_j \}. \]
Then $|G| \leq L(N, x_1, \dots, x_n)$. We now show that $G$ is an $\epsilon$-net of $\mathbf{Q} = \{\left(Q(x_1, y_1, f), \dots, Q(x_n, y_n, f)\right), f \in \mathcal{F}\}$. For each sample $i \in \{1, \dots, n\}$, find $k_i$ such that $f(x_i) \in [q_{k_i}, q_{k_i + 1})$. Set $g_i = q_{k_i}$. Now,
\begin{align*}
\left| (y_i - f(x_i))^2 - (y_i - q_{k_i})^2 \right| &= \left| y_i^2 - 2 y_i f(x_i) + f(x_i)^2 - y_1^2 + 2y_i q_{k_i} - q_{k_i}^2 \right|\\
&= \left| f(x_i)^2 + 2y_i\left( q_{k_i} - f(x_i)\right) - q_{k_i}^2\right|\\
&= \left| (f(x_i) - q_{k_i}) \left( f(x_i) + q_{k_i}\right) - 2y_i \left( f(x_i) - q_{k_i}\right) \right|\\
&= \left( f(x_i) - q_{k_i}\right) \left| f(x_i) + q_{k_i} - 2y_i \right|\\
&\leq \frac{2}{N}\\
&= \frac{2}{\left \lceil \frac{2}{\epsilon} \right \rceil}\\
&\leq \epsilon
\end{align*}
It remains to show that $x_i \preceq x_j \implies g_i \leq g_j$. Since $f$ is coordinate-wise monotone, $x _i \preceq x_j \implies f(x_i) \leq f(x_j)$. Then also $g_i \leq g_j$. Therefore, we have shown that $G$ is a valid $\epsilon$-net, and we conclude that the size of the smallest $\epsilon$-net is at most $L(N, x_1, \dots, x_n)$.
\end{proof}

The $m$-labeling number is in turn related to the binary labeling number. 
\begin{proposition}\label{prop:binary-labeling}\cite{Gamarnik1999}
It holds that 
$L(m, x_1, \dots, x_n) \leq \left( L(2, x_1, \dots, x_n)\right)^{m-1}.$
\end{proposition}
\begin{proof}
The proof can be found in the proof of Lemma 3 in \cite{Gamarnik1999}, with the correction that
\[g_2(x_i) = \begin{cases} 1 & \text{ if } g(x_i) \leq m\\ 2 & \text{ if } g(x_i) = m+ 1. \end{cases} \]
\end{proof}

Propositions \ref{prop:labeling-number} and \ref{prop:binary-labeling} suggest that the binary labeling number is a good proxy for the VC entropy. 

\begin{theorem}\label{thm:labeling-number-bounds}
Let $X_1, \dots, X_n$ be distributed uniformly and independently in $[0,1]^d$. Let $L(X_1, \dots, X_n)$ be the number of binary monotone labelings of the points $X_1, \dots, X_n$. Then \[ \exp \left[\frac{\log(2) (1- e^{-1})}{(d-1)!}n^{\frac{d-1}{d}} \right] \leq \mathbb{E}[L(X_1, \dots X_n)] \leq \exp \left[ \left(2^d + 2\log(2) -1 \right) n^{\frac{d-1}{d}}\right].\]
\end{theorem}

In order to prove the upper bound in Theorem \ref{thm:labeling-number-bounds}, we relate the binary labeling number to the number of \emph{integer partitions}.
\begin{definition}[Integer Partition]
An integer partition of dimension $(d-1)$ with values in $\{0, 1, \dots m\}$, is a collection of values $A_{i_1, i_2, \dots, i_{d-1}} \in \{0, 1, \dots, m\}$ where $i_k \in \{1, \dots m\}$ and $A_{i_1, i_2, \dots, i_{d-1}} \leq A_{j_1, j_2, \dots, j_{d-1}}$ whenever $i_k \leq j_k$ for all $k \in \{1, \dots, d-1\}$. The set of integer partitions of dimension $(d-1)$ with values in $\{0, 1, \dots m\}$ is denoted by $P([m]^d)$.
\end{definition}
Note: the definition is in terms of $(d-1)$ because when the monotone regression problem is in dimension $d$, we will consider partitions of dimension $(d-1)$. To illustrate the definition, consider setting $d=2$ (see Figure \ref{fig:integer-partition-illustration}). An integer partition of dimension $1$ is an assignment of values $(A_1, A_2, \dots, A_m)$ that is non-increasing, and each $A_k$ takes value in $\{0, 1, \dots, m\}$. A $1$-dimensional partition can be seen to divide the $m \times m$ grid in a monotonic way. Next we define the concept of a \emph{border cell}.

\begin{definition}[Border Cell]
Label the cells in the $[m]^d$ grid according to cell coordinates, namely entries $(x_{1}, x_{2}, \dots x_{d})$, where $x_k \in \{1, \dots, m\}$ for each $k \in \{1, \dots, d\}$. For a partition $p \in P([m]^d)$ with entries in $\{1, \dots, m\}$, consider its values $A_{i_1, i_2, \dots, i_{d-1}}$. The cells corresponding to the partition (which we call the partition cells) are given by $(x_1, x_2, \dots, x_{d-1}, x)$, for $x \leq A_{x_1, x_2, \dots, x_{d-1}}$ and where each $x_k$ ranges in $\{1, \dots, m\}$. We say that two cells are adjacent if they share a face or a corner. The border cells are defined to be the partition cells that are adjacent to at least one cell that is not a partition cell.
\end{definition}

\begin{figure}
\begin{center}
\begin{tikzpicture}
\tikzset{cross/.style={cross out, draw=black, minimum size=2*(#1-\pgflinewidth), inner sep=0pt, outer sep=0pt}, cross/.default={4pt}};
\draw[step=0.3,black,thin] (0,0) grid (3,3);
\node[fill=gray] at (0.15,0.15) {};
\node[fill=gray] at (0.15,0.45) {};
\node[fill=gray] at (0.15,0.75) {};
\node[fill=gray] at (0.15,1.05) {};
\node[fill=gray] at (0.15,1.35) {};
\node[fill=gray] at (0.15,1.65) {};
\node[fill=gray] at (0.15,1.95) {};
\node[fill=gray] at (0.15,2.25) {};
\node[fill=gray] at (0.15,2.55) {};
\node[fill=gray] at (0.45,0.15) {};
\node[fill=gray] at (0.45,0.45) {};
\node[fill=gray] at (0.45,0.75) {};
\node[fill=gray] at (0.45,1.05) {};
\node[fill=gray] at (0.45,1.35) {};
\node[fill=gray] at (0.45,1.65) {};
\node[fill=gray] at (0.45,1.95) {};
\node[fill=gray] at (0.75,0.15) {};
\node[fill=gray] at (0.75,0.45) {};
\node[fill=gray] at (0.75,0.75) {};
\node[fill=gray] at (0.75,1.05) {};
\node[fill=gray] at (0.75,1.35) {};
\node[fill=gray] at (0.75,1.65) {};
\node[fill=gray] at (1.05,0.15) {};
\node[fill=gray] at (1.05,0.45) {};
\node[fill=gray] at (1.05,0.75) {};
\node[fill=gray] at (1.05,1.05) {};
\node[fill=gray] at (1.05,1.35) {};
\node[fill=gray] at (1.05,1.65) {};
\node[fill=gray] at (1.35,0.15) {};
\node[fill=gray] at (1.35,0.45) {};
\node[fill=gray] at (1.35,0.75) {};
\node[fill=gray] at (1.35,1.05) {};
\node[fill=gray] at (1.35,1.35) {};
\node[fill=gray] at (1.35,1.65) {};
\node[fill=gray] at (1.65,0.15) {};
\node[fill=gray] at (1.65,0.45) {};
\node[fill=gray] at (1.65,0.75) {};
\node[fill=gray] at (1.65,1.05) {};
\node[fill=gray] at (1.95,0.15) {};
\node[fill=gray] at (1.95,0.45) {};
\node[fill=gray] at (1.95,0.75) {};
\node[fill=gray] at (2.25,0.15) {};
\node[fill=gray] at (2.25,0.45) {};
\node[fill=gray] at (2.55,0.15) {};
\node[fill=gray] at (2.85,0.15) {};
\draw (0.15,1.95) node[cross] {};
\draw (0.15,2.25) node[cross] {};
\draw (0.15,2.55) node[cross] {};
\draw (0.45,1.65) node[cross] {};
\draw (0.45,1.95) node[cross] {};
\draw (0.75,1.65) node[cross] {};
\draw (1.05,1.65) node[cross] {};
\draw (1.35,1.05) node[cross] {};
\draw (1.35,1.35) node[cross] {};
\draw (1.35,1.65) node[cross] {};
\draw (1.65,0.75) node[cross] {};
\draw (1.65,1.05) node[cross] {};
\draw (1.95,0.45) node[cross] {};
\draw (1.95,0.75) node[cross] {};
\draw (2.25,0.15) node[cross] {};
\draw (2.25,0.45) node[cross] {};
\draw (2.55,0.15) node[cross] {};
\draw (2.85,0.15) node[cross] {};
\end{tikzpicture}
\end{center}
\caption{Illustration of a partition in $d=2$ with $m = 10$. The partition cells are indicated in gray, and the border cells are marked.}\label{fig:integer-partition-illustration}
\end{figure}
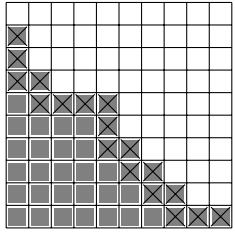
\begin{lemma}\label{lemma:border-cells}
The number of border cells in any $(d-1)$-dimensional integer partition with entries from $\{1, \dots, m\}$ is at most $m^{d} - (m-1)^d$.
\end{lemma}
\begin{proof}
When $d=2$, the number of cells on the border of any ($1$-dimensional) partition with values in $\{1, \dots, m\}$ is at most $2m -1 = m^2 - (m-1)^2$, corresponding to a path from $(1,m)$ to $(m,1)$. When $d=3$, the number of border cells in any ($2$-dimensional) partition with values in $\{1, \dots, m\}$ is at most corresponding to border cells that include $(1,m,m)$ and $(m,1,1)$. All partitions with such border cells have the same number of border cells. The simplest of these is the one where each cell is on the perimeter of the cube. The number of border cells in such a partition is equal to $3m^2 - 3m + 1 = m^3 - (m-1)^3$. For general $d$, the number of border cells in a $(d-1)$-dimensional partition taking values in $\{1, \dots, m\}$ is upper bounded by the total number of cells minus the number of cells in an $(m-1)^d$ grid, in other words, $m^{d} - (m-1)^d$. 
\end{proof}
The key idea of the proof of the upper bound in Theorem \ref{thm:labeling-number-bounds} comes from the following lemma.
\begin{lemma}\label{lemma:border-cells-upper-bound}
Let $N \sim \text{Binom}\left(n, \frac{m^{d} - (m-1)^d}{m^d}\right)$. It holds that
\begin{align*}
\mathbb{E}[L(X_1, \dots, X_n)] &\leq \left|P\left([m]^d\right) \right| \mathbb{E}[2^N].
\end{align*}
\end{lemma}
\begin{proof}
The idea of the proof comes from the proof of Theorem 13.13 in \cite{Devroye1996}, who showed a similar result for $d=2$. Consider a binary coordinate-wise monotone function $f$, with domain $[0,1]^d$. Let $S_0 = \{x \in [0,1]^d : f(x) = 0\}$ and $S_1 = \{x \in [0,1]^d : f(x) = 1\}$. The number of binary labelings of a set of points $X_1, \dots, X_n$ is equal to the number of partitions $(S_0, S_1)$ producing distinct labelings. To upper-bound the number of dividing surfaces, we divide the $d$-dimensional cube into an $m^d$ grid, $[m]^d$. That is, each cell in the grid has side length $\frac{1}{m}$. Let $B$ be the intersection of the boundaries of the $S_0$ and $S_1$. For example, if 
\begin{align*}
f(x) &= \begin{cases}
0 &\text{ if } x_ 1 + x_2 <1\\
1 &\text{ if } x_1 + x_2 \geq 1
\end{cases}
\end{align*}
then $B = \{x : x_1 + x_2 = 1\}$. Now consider the subset of cells that contain at least one element of $B$. These cells are necessarily the border cells of some $(d-1)$-dimensional integer partition with values from $\{1, \dots, m\}$. Therefore, we can upper bound the number of labelings as follows. For a boundary $B$ corresponding to a partition $(S_0, S_1)$, add a contribution of $2^{N_B}$, where $N_B$ is the number of points within the border cells containing the elements of $B$. This contribution corresponds to all (valid or invalid) labelings of the points within the border cells. Points outside the border cells are labeled $0$ if they fall in $S_0$ and $1$ if they fall in $S_1$. From Lemma \ref{lemma:border-cells}, the number of points in the border cells of a partition with the maximal number of border cells is distributed as a binomial random variable $N$ with parameters $\left(n, \frac{m^{d} - (m-1)^d}{m^d}\right)$. The expected labeling number is therefore upper bounded by $\left|P\left([m]^d\right) \right| \cdot  \mathbb{E}[2^N]$.
\end{proof}

\begin{proof}[Proof of Theorem \ref{thm:labeling-number-bounds}]
\textbf{Upper bound}\\
From Lemma \ref{lemma:border-cells-upper-bound}, we know that
\[\mathbb{E}[L(X_1, \dots, X_n)] \leq \left|P\left([m]^d\right) \right| \cdot \mathbb{E}[2^N].\]
Now, 
\begin{align*}
\mathbb{E}[2^N] =  \mathbb{E}[e^{\log(2)N}] = M_N(\log(2)),
\end{align*}
where $M_N(\cdot)$ is the moment-generating function of the random variable $N$. A binomial random variable $Z$ with parameters $(n,p)$ has moment-generating function $M_Z(\theta) = (1-p + pe^{\theta})^n$. Additionally, \cite{Moshkovitz2014} showed that \[\left|P\left([m]^d\right) \right| \leq \binom{2m}{m}^{m^{d-2}}.\] Substituting, 
\begin{align*}
\mathbb{E}[L(X_1, \dots, X_n)] &\leq \binom{2m}{m}^{m^{d-2}} \left(1-\frac{m^{d} - (m-1)^d}{m^d} + \frac{m^{d} - (m-1)^d}{m^d}e^{\log(2)}\right)^n\\
&= \binom{2m}{m}^{m^{d-2}}\left(1 + \frac{m^{d} - (m-1)^d}{m^d}\right)^n\\
&\leq 2^{2m \cdot m^{d-2}}\left(1 + \frac{m^{d} - (m-1)^d}{m^d}\right)^n\\
&=\exp\left[2\log(2) m^{d-1}+ n \log\left(1 + \frac{m^{d} - (m-1)^d}{m^d}\right) \right]
\end{align*}
Choosing $m = n^{\frac{1}{d}}$, 
\begin{align*}
\mathbb{E}[L(X_1, \dots, X_n)] &\leq \exp\left[2\log(2)n^{\frac{d-1}{d}}+ n \log\left(1 + \frac{n - \left(n^{\frac{1}{d}} - 1\right)^d}{n}\right) \right]
\end{align*}
Since $\log(1+x) \leq x$,
\begin{align*}
\mathbb{E}[L(X_1, \dots, X_n)] &\leq \exp\left[2\log(2)n^{\frac{d-1}{d}} +  n - \left(n^{\frac{1}{d}} - 1\right)^d \right]
\end{align*}
Applying the Binomial Theorem,
\begin{align*}
n -\left(n^{\frac{1}{d}} - 1\right)^d &= n - \sum_{k=0}^d \binom{d}{k} n^{\frac{d-k}{d}} (-1)^k\\
&=  -\sum_{k=1}^d \binom{d}{k}n^{\frac{d-k}{d}} (-1)^k\\
&\leq \sum_{k=1}^d \binom{d}{k} \cdot \max_{k \in \{1, \dots, d\}} n^{\frac{d-k}{d}} (-1)^{k+1}\\
&= \left(2^d -1\right)n^{\frac{d-1}{d}}
\end{align*}
Substituting, we obtain
\begin{align*}
\mathbb{E}[L(X_1, \dots, X_n)] &\leq \exp\left[2\log(2)n^{\frac{d-1}{d}} +  \left(2^d -1\right)n^{\frac{d-1}{d}}\right]\\
&= \exp \left[ \left(2^d + 2\log(2) -1 \right) n^{\frac{d-1}{d}}\right]
\end{align*}
\textbf{Lower Bound}\\
Let $N$ be an integer, which will be specified later. Divide $[0,1]^d$ into $N^d$ cells of side length $\frac{1}{N}$. The cells are labeled in the natural coordinate system, writing $C = (x_1, \dots, x_d) \in [N]^d$. We say that two cells are \emph{incomparable} if for all $x \in C_1$ and $y \in C_2$, neither $x \preceq y$ nor $x \succeq y$. 
Let us find the number of incomparable cells. 
\begin{lemma}\label{lemma:incomparable-cells}
The number of incomparable cells is at least $\binom{N + d -2}{d-1}$.
\end{lemma}
\begin{proof}
Consider any two cells $C_1 = (x_1, x_2, \dots, x_d)$ and $C_2 = (y_1, y_2, \dots, y_d)$. If $\sum_{i=1}^d x_i = \sum_{i=1}^d y_i$, then either $(x_1, \dots, x_d)= (y_1, \dots, y_d)$ or $(x_1, \dots, x_d) \not \preceq (y_1, \dots, y_d)$ and $(x_1, \dots, x_d) \not \succeq (y_1, \dots, y_d)$. Observe that if $(x_1, \dots, x_d) \not \preceq (y_1, \dots, y_d)$ and $(x_1, \dots, x_d) \not \succeq (y_1, \dots, y_d)$, then $C_1$ and $C_2$ are incomparable. In dimension $d$, let us therefore count the number of cells whose coordinates sum to $N + d -1$. This corresponds to the number of integer compositions of $(N+d -1)$ into $d$ parts, which is given by $\binom{N + d -2}{d-1}$.
\end{proof}

The number of incomparable points, $Y_n$ is at least the number of occupied incomparable cells, which we call $\Delta$. For $d \geq 2$,
\begin{align*}
\mathbb{E}[\Delta] &\geq \binom{N + d -2}{d-1} \left(1 - \left(1-\frac{1}{N^d}\right)^n\right)\\
&= \frac{(N+d-2)!}{(d-1)! (N-1)!}\left(1 - \left(1-\frac{1}{N^d}\right)^n\right)\\
&\geq \frac{N^{d-1}}{(d-1)!}\left(1 - \left(1-\frac{1}{N^d}\right)^n\right)
\end{align*}
Now let $N = \left \lceil n^{\frac{1}{d}} \right \rceil$. Then
\begin{align*}
\mathbb{E}[\Delta] &\geq \frac{n^{\frac{d-1}{d}}}{(d-1)!}  \left(1- e^{-1}\right)
\end{align*}
We can now lower bound the labeling number. By Jensen's inequality,
\begin{align*}
\mathbb{E}\left[L(X)\right] &\geq \mathbb{E} \left[2^{\Delta}\right] \geq 2^{\mathbb{E}[\Delta]} \geq 2^{\frac{1-e^{-1}}{(d-1)!} n^{\frac{d-1}{d}}} = \exp \left[\frac{\log(2) (1- e^{-1})}{(d-1)!}n^{\frac{d-1}{d}} \right]
\end{align*}
\end{proof}

Finally, we tie together the above results to prove Theorem \ref{thm:statistical-consistency}.
\begin{proof}[Proof of Theorem \ref{thm:statistical-consistency}]
The proof is by chaining the inequalities from Propositions \ref{prop:statistical-consistency}- \ref{prop:binary-labeling}, along with Theorem \ref{thm:labeling-number-bounds}.
\end{proof}

\begin{proof}[Proof of Corollary \ref{corollary:noisy-output-statistical-consistency}]
Equivalently, we show that $s = o\left(\sqrt{\log(n)}\right)$ and $d = e^{o \left(\frac{n}{s}\right)}$ suffices. Analyzing the leading term in the exponent,
\begin{align*}
2^s n^{\frac{s-1}{s}} &= n^{1 + s \frac{\log(2)}{\log(n)} - \frac{1}{s}}
\end{align*}
Analyzing the exponent,
\begin{align*}
1 + s \frac{\log(2)}{\log(n)} - \frac{1}{s}&= 1 + \frac{o \left( \sqrt{\log(n)}\right)}{\log(n)} - \frac{1}{o \left(\sqrt{\log(n)} \right)}\\
&= 1 + o \left( \frac{1}{\sqrt{\log(n)}}\right) - \omega \left( \frac{1}{\sqrt{\log(n)}}\right) \\
&= 1  - \omega \left( \frac{1}{\sqrt{\log(n)}}\right)
\end{align*}
Therefore, 
\begin{align*}
\exp\left\{\left(\left \lceil \frac{2^{11}}{\epsilon^2} \right \rceil - 1 \right) \left(2^s + 2\log(2) -1 \right) n^{\frac{s-1}{s}}  {-\frac{3 \epsilon^3 n}{41 \times 2^{10}}} \right\} & = \exp\left\{\Theta(1)n^{1  - \omega \left( \frac{1}{\sqrt{\log(n)}}\right)} - \Theta(n) \right \}\\
&= \exp \left\{ \Theta(n) \left( n^{- \omega \left( \frac{1}{\sqrt{\log(n)}}\right)} - 1\right) \right \}\\
&= \exp \left\{ \Theta(n) \left( e^{- \omega \left( \frac{1}{\sqrt{\log(n)}}\right) \log(n)} - 1\right) \right \}\\
&= \exp \left\{ \Theta(n) \left( e^{- \omega \left( \sqrt{\log(n)}\right)}- 1\right) \right \}\\
&= \exp \left\{ \Theta(n) \left( o \left(e^{-\sqrt{\log(n)}}\right)- 1\right) \right \}\\
& \exp \left\{ - \Theta(n)\right\}
\end{align*}

Next, 
\begin{align*}
\binom{d}{s} &\leq d^s\\
&= e^{s \log(d)}
\end{align*}
We need $s \log(d) = o(n)$, or equivalently, $d = e^{o  \left( \frac{n}{s}\right)}$.
\end{proof}

To prove Theorem \ref{thm:multiple-coordinates}, we first give guarantees for the recovery of a single coordinate. For that, we need Lemma \ref{lemma:stochastically-increasing}.
\begin{lemma}\label{lemma:stochastically-increasing}
It holds that $p_{k} > 0$. In other words, when $X_1$ is greater than $X_2$ in at least one active coordinate the output is more likely to be larger than smaller.
\end{lemma}

\begin{proof}
Consider the following procedure. We sample $X_1$ and $X_2$ independently and uniformly on $[0,1]^d$. Fix $k \in A$. Let
\begin{align*}
X_+ = \begin{cases}
X_1 & \text{ if } X_{1,k} > X_{2,k}\\
X_2 &\text{ otherwise}
\end{cases}
\end{align*}
and
\begin{align*}
X_- = \begin{cases}
X_1 & \text{ if } X_{1,k} \leq X_{2,k}\\
X_2 &\text{ otherwise.}
\end{cases}
\end{align*}
In other words, $X_+$ is the right point according to coordinate $k$ and $X_-$ is the left point according to the same coordinate. Now, 
\begin{align*}
&\mathbb{P} \left( f(X_1) + W_1 > f(X_2) + W_2 | X_{1,k} > X_{2,k}  \right) \\
&= \mathbb{P} \left( f(X_1) + W_1 > f(X_2) + W_2 | X_{1,k} = X_+, X_{2,k} = X_-  \right)\\
&= \mathbb{P} \left( f(X_+) + W_1 > f(X_-) + W_2 | X_{1,k} = X_+, X_{2,k} = X_-  \right)\\
&= \mathbb{P} \left( f(X_+) + W_1 > f(X_-) + W_2  \right)
\end{align*}
Similarly, 
\begin{align*}
\mathbb{P} \left( f(X_1)  W_1 < f(X_2) + W_2 | X_{1,k} > X_{2,k}  \right) &= \mathbb{P} \left( f(X_+) + W_1 < f(X_-) + W_2  \right)
\end{align*}
Therefore, we can equivalently define $p_k$ as
\[p_{k} = \mathbb{P} \left( f(X_+) + W_1 > f(X_-) + W_2 \right) - \mathbb{P} \left( f(X_+) + W_1 < f(X_-) + W_2 \right) .\]
Our goal is to show that
\begin{align*}
&\mathbb{P} \left( f(X_+) + W_1 > f(X_-) + W_2 \right) > \mathbb{P} \left( f(X_+) + W_1 < f(X_-) + W_2 \right) 
\end{align*}
Due to the monotonicity of $f$ with respect to $A \ni k$, it holds that
\begin{align*}
&\mathbb{P} \left( f(X_+)  > f(X_-) \right) > \mathbb{P} \left( f(X_+)  < f(X_-) \right).
\end{align*}
Coupling $W_1$ and $W_2$ across the events $\{f(X_+)  > f(X_-)\}$ and $\{f(X_+)  < f(X_-)\}$, we obtain 
\begin{align*}
&\mathbb{P} \left( f(X_+) + W_1 > f(X_-) + W_2 \right) > \mathbb{P} \left( f(X_+) + W_1 < f(X_-) + W_2 \right).
\end{align*}
\end{proof}
We now give a guarantee for the recovery of a single coordinate.
\begin{lemma}\label{lemma:one-coordinate}
Suppose $s=1$, and assume without loss of generality that $A = \{1\}$. Algorithm \ref{alg:sequential-recovery} (which is equivalent to Algorithm \ref{alg:simultaneous-recovery} when $s=1$) recovers the correct coordinate with probability at least
\[1 -  (d-1) \exp\left(-\frac{n p_{1}^2}{16} \right).\]
\end{lemma}

\begin{proof}
For a fixed value of $v = \overline{v}$, the optimal choice is to set 
$$\sum_{k=1}^d q(i,j,k) c^{ij}_k = \max \left\{0, 1 - \sum_{k=1}^d q(i,j,k) \overline{v}_k \right \}$$ for $i,j$ such that $Y_i > Y_j$ and $\sum_{k=1}^d q(i,j,k) \geq 1$, with $c_k^{ij} = 0$ whenever $q(i,j,k) = 0$. Note that $\sum_{k=1}^d q(i,j,k) c^{ij}_k  = \sum_{k=1}^d c^{ij}_k$. Therefore, the objective function is equal to 
\begin{align*}
z(\overline{v}) &\triangleq \sum_{i=1}^n \sum_{j=1}^n \mathbbm{1} \left\{Y_i > Y_j, X_i \not \preceq X_j\right\}\max \left\{0, 1 - \sum_{k=1}^d q(i,j,k) \overline{v}_k \right \}.
\end{align*}
Let $v^{\star} = e_1$ be the indicator for the active coordinate. For a particular coordinate $j \neq 1$, we will consider all feasible solutions $\overline{v} = v^{\star} + u$, where $\overline{v}_j  = \max_{k \in \{1, \dots, d\}} \overline{v}$. We will show that $z(v^{\star}) < z(\overline{v})$ for all such $\overline{v}$, with high probability. Now,
\small
\begin{align*}
&z(\overline{v}) - z(v^{\star})  \\
&= \sum_{i=1}^n\sum_{j=1}^n \mathbbm{1} \left\{Y_i > Y_j, X_i \not \preceq X_j\right\} \left(\max \left\{0, 1 - \sum_{k=1}^d q(i,j,k) \overline{v}_k \right \} - \max \left\{0, 1 - \sum_{k=1}^d q(i,j,k) v^{\star}_k \right \}  \right).
\end{align*}
\normalsize
We have
\begin{align*}
 \max \left\{0, 1 - \sum_{k=1}^d q(i,j,k) v^{\star}_k \right \} &= \begin{cases}
 0 & \text{if } q(i,j,1) = 1\\
 1 & \text{if } q(i, j, 1) = 0
\end{cases}
\end{align*}
and
\begin{align*}
 \max \left\{0, 1 - \sum_{k=1}^d q(i,j,k) \overline{v}_k \right \} &= \begin{cases}
 -u_1 - \sum_{k \neq 1} q(i,j,k)u_k & \text{if } q(i,j,1) = 1\\
 1 - \sum_{k \neq 1} q(i,j,k)u_k & \text{if } q(i, j, 1) = 0
\end{cases}
\end{align*}
Therefore, 
\begin{align*}
&\max \left\{0, 1 - \sum_{k=1}^d q(i,j,k) \overline{v}_k \right \} -  \max \left\{0, 1 - \sum_{k=1}^d q(i,j,k) v^{\star}_k \right \}\\
 &= \begin{cases}
 -u_1 - \sum_{k \neq 1} q(i,j,k)u_k & \text{if } q(i,j,1) = 1\\
 - \sum_{k \neq 1} q(i,j,k)u_k & \text{if } q(i, j, 1) = 0
\end{cases}\\
&= - \sum_{k=1}^d q(i,j,k)u_k
\end{align*}
Substituting,
\begin{align*}
z(\overline{v}) - z(v^{\star}) &= -\sum_{i=1}^n \sum_{j=1}^n \mathbbm{1} \left\{Y_i > Y_j, X_i \not \preceq X_j\right\} \left(\sum_{k=1}^d q(i,j,k)u_k \right)\\
&= - \sum_{k=1}^p u_k \sum_{i=1}^n \sum_{j=1}^n \mathbbm{1} \left\{Y_i > Y_j, X_i \not \preceq X_j\right\}  q(i,j,k)\\
&= - \sum_{k=1}^d u_k \sum_{i=1}^n \sum_{j=1}^n \mathbbm{1} \left\{Y_i > Y_j\right\}  q(i,j,k)\\
&= -u_1 \sum_{i=1}^n \sum_{j=1}^n \mathbbm{1} \left\{Y_i - Y_j > t\right\}  q(i,j,1) - \sum_{k \neq 1} \sum_{i=1}^n \sum_{j=1}^n \mathbbm{1} \left\{Y_i - Y_j > t\right\}  q(i,j,k)\\
&= \sum_{k \neq 1} u_k \sum_{i=1}^n \sum_{j=1}^n \mathbbm{1} \left\{Y_i > Y_j\right\}  q(i,j,1) - \sum_{k \neq 1} u_k\sum_{i=1}^n \sum_{j=1}^n \mathbbm{1} \left\{Y_i > Y_j\right\}  q(i,j,k)\\
&= \sum_{k \neq 1} u_k \sum_{i=1}^n \sum_{j=1}^n \mathbbm{1} \left\{Y_i > Y_j\right\}  \left(q(i,j,1) - q(i,j,k)\right)
\end{align*}
We show that this quantity is greater than zero with high probability, by concentration.
\small
\begin{align*}
&\mathbb{E}_{X, W} \left[z(\overline{v}) - z(v^{\star}) \right] \\
&= \sum_{k \neq 1} u_k \sum_{i=1}^n \sum_{j=1}^n \mathbb{E} \left[\mathbbm{1} \left\{Y_i > Y_j\right\}  \left(q(i,j,1) - q(i,j,k)\right) \right]\\
&= \sum_{k \neq 1} u_k \sum_{i=1}^n \sum_{j=1}^n \left(\mathbb{P} \left( Y_i > Y_j | q(i,j,1) = 1, q(i,j,k) = 0 \right) - \mathbb{P} \left( Y_i > Y_j | q(i,j,1) = 0, q(i,j,k) = 1 \right) \right)\\
&=  n(n-1)\sum_{k \neq 1} u_k  \left(\mathbb{P} \left( Y_1 > Y_2 | q(1,2,1) = 1, q(1,2,k) = 0 \right) - \mathbb{P} \left( Y_1 > Y_2 | q(1,2,1) = 0, q(1,2,k) = 1 \right) \right)\\
&=  n(n-1)\sum_{k \neq 1} u_k  \left(\mathbb{P} \left( Y_1 > Y_2 | q(1,2,1) = 1, q(1,2,k) = 0 \right) - \mathbb{P} \left( Y_1 > Y_2 | q(2,1,1) = 1, q(2,1,k) = 0 \right) \right)\\
&=  n(n-1)\sum_{k \neq 1} u_k  \left(\mathbb{P} \left( Y_1 > Y_2 | q(1,2,1) = 1, q(1,2,k) = 0 \right) - \mathbb{P} \left( Y_2 > Y_1 | q(1,2,1) = 1, q(1,2,k) = 0 \right) \right)
\end{align*}
\normalsize
Note that we can drop the conditioning on $q(1,2,k) = 0$ because of the uniform distribution of $X_1$ and $X_2$. Continuing,
\begin{align*}
\mathbb{E}_{X, W} \left[z(\overline{v}) - z(v^{\star}) \right] &=  n(n-1)\sum_{k \neq 1} u_k  \left(\mathbb{P} \left( Y_1 > Y_2 | q(1,2,1) = 1 \right) - \mathbb{P} \left( Y_2 > Y_1 | q(1,2,1) = 1 \right) \right)\\
&=  n(n-1) p_1 \sum_{k \neq 1} u_k  \\
&\geq \frac{p_1}{2} n(n-1)
\end{align*}
Observe that changing any one of the $X_i$ or $W_i$ variables can change the value of $z(\overline{v}) - z(v^{\star})$ by at most $2(n-1)$ in absolute value. Applying the McDiarmid inequality gives
\begin{align*}
\mathbb{P} \left( z(\overline{v}) - z(v^{\star}) \leq 0 \right)& \leq \exp\left(-2 \frac{\left(\mathbb{E}_{X, W} \left[z(\overline{v}) - z(v^{\star}) \right]\right)^2}{2n(2(n-1))^2} \right) \\
&\leq \exp\left(-2 \frac{\frac{1}{4}n^2(n-1)^2 p_{1}^2}{8 n(n-1)^2} \right)\\
&= \exp\left(-\frac{np_{1}^2}{16} \right).
\end{align*}
Using the Union Bound, the probability that the optimal solution to the LP gives $A(v) \neq \{1\}$ is at most
 \[  (d-1)\exp\left(-\frac{np_{1}^2}{16} \right). \]
\end{proof}

\begin{proof}[Proof of Theorem \ref{thm:multiple-coordinates}]

First, we show that $k_1 \in \{1, \dots, s\}$ with high probability. We show that for all $i \in \{1, \dots, s\}$ and $j \not \in \{1, \dots, s\}$, any solution $\overline{v}$ such that $\overline{v}_j = \max_{k \in \{1, \dots, d\}} \overline{v}_k$ satisfies $z(\overline{v}) > z(e_i)$ with high probability. 

Let $v^{\star} = e_i$ and write $\overline{v} = v^{\star} + u$. Then adapting the result of Lemma \ref{lemma:one-coordinate}, it holds that
\begin{align*}
\mathbb{P} \left( z(\overline{v}) - z(v^{\star}) \leq 0 \right)&\leq \exp\left(-\frac{n p_{i}^2}{16} \right)
\end{align*}
Therefore, the probability of an error in the coordinate $k_1$ is at most
\[ (d-s)\sum_{i =1}^s \exp\left(-\frac{n p_{i}^2}{16} \right).\]
Now condition on the correctness of $k_1$. We show that for all $i \in \{1, \dots, s\} \setminus \{k\}$ and $j \not \in \{1, \dots, s\}$, the coordinate $k_2$ is correct with high probability. Repeating the argument, the probability of an error in coordinate $k_2$ conditioned on $k_1$ being correct is at most 
\[ (d-s)\sum_{1 \leq i \leq s, i \neq k_1} \exp\left(-\frac{n p_{i}^2}{16} \right).\]
Continuing the analysis, the probability that $k_i$ is incorrect given that $k_1, \dots, k_{i-1}$ are correct is upper bounded by 
\[ (d-s) \sum_{1 \leq i \leq s, i \neq k_1, \dots, k_{i-1}} \exp\left(-\frac{np_{i}^2}{16} \right).\]
Recalling the assumption that $p_1 \leq \dots \leq p_s$, we conclude that the probability that $B \neq A$ is upper bounded by 
\begin{align*}
(d-s) \sum_{k=1}^s (s+1 -k) \exp\left(-\frac{n p_{k}^2}{16} \right)
\end{align*}
Therefore, $B = A$ with probability at least 
\[1 - (d-s)\sum_{k=1}^s (s+1 -k) \exp\left(-\frac{n p_{k}^2}{16} \right) \]
\end{proof}

\begin{proof}[Proof of Corollary \ref{corollary:recovery}]
\begin{align}
(d-s)\sum_{k=1}^s (s+1 -k) \exp\left(-\frac{n p_{k}^2}{16} \right) &\leq d^3 e^{-\frac{n p_1^2}{16}} \nonumber \\
&= e^{3 \log(d) - \frac{n p_1^2}{16}} \label{eq:expression}
\end{align}
Therefore, if $n = \omega(\log(d))$, then \eqref{eq:expression} goes to zero.
\end{proof}

\begin{proof}[Proof of Corollary \ref{corollary:two-stage-noisy-output}]
Support recovery fails with probability at most  \[(d-s)\sum_{k=1}^s (s+1 -k) \exp\left(-\frac{n p_{k}^2}{16} \right). \] If it succeeds, the probability of the $L_2$ norm error exceeding $\delta$ is upper bounded by the value in Theorem \ref{thm:statistical-consistency}, with $d$ set to $s$. Then $\mathbb{P} \left( \Vert \hat{f}_n - f \Vert_2^2 > \delta \right)$ is at most
\begin{align*}
&(d-s)\sum_{k=1}^s (s+1 -k) \exp\left(-\frac{n p_{k}^2}{16} \right)\\
&+ 6\exp\left\{\left(\left \lceil \frac{2^{11}}{\epsilon^2} \right \rceil - 1 \right) \left(2^s + 2\log(2) -1 \right) n^{\frac{s-1}{s}}  {-\frac{3 \epsilon^3 n}{41 \times 2^{10}}} \right\}.
\end{align*}
Therefore, if $n = \omega(\log(d))$ and $n = e^{\omega(s^2)}$, the estimator is consistent, by Corollaries \ref{corollary:noisy-output-statistical-consistency} and \ref{corollary:recovery}.
\end{proof}

\clearpage
\section{Proofs for the Noisy Input Model}

\begin{proof}[Proof of Theorem \ref{thm:consistency-binary}]
To illustrate the proof idea, we show the claim for $d = s =1$ first. Observe that for any monotone partition $(S_0, S_1)$ in $\mathbb{R}$, either $S_0 = \{x : x \leq r\}$ or $S_0 = \{x : x < r\}$ for some $r$. When $d = s =1$, the optimization problem \eqref{eq:simultaneous-first}-\eqref{eq:simultaneous-last} amounts to finding a boundary $r \in \mathbb{R}$.
Let
\[g\left(X_{1:n}, W_{1:n}; (S_0, S_1) \right) = \sum_{i=1}^n \mathbbm{1}\left\{f(X_i + W_i) = 1, X_i  \in S_0 \right\} + \mathbbm{1}\left\{f(X_i + W_i) = 0, X_i \in S_1 \right\}  \]
denote the corresponding value of the objective function. Observe that the value of $g\left(X_{1:n}, W_{1:n} ; (S_0, S_1)\right)$ can change by at most $\pm 2$ when any one of the random variables is changed. Applying the McDiarmid inequality, for all $\epsilon > 0$, it holds that
\begin{align*}
\mathbb{P} \left( g\left(X_{1:n}, W_{1:n}; (S_0, S_1) \right) - \mathbb{E} \left[g\left(X_{1:n}, W_{1:n}; (S_0, S_1) \right) \right]  \geq \epsilon n \right) &\leq  \exp\left(- \frac{2 \epsilon^2 n^2}{2n \cdot 2^2} \right)\\
&=  \exp\left(- \frac{ \epsilon^2 n}{4} \right).
\end{align*}
Similarly, 
\begin{align*}
\mathbb{P} \left( g\left(X_{1:n}, W_{1:n}; (S_0, S_1) \right) - \mathbb{E} \left[g\left(X_{1:n}, W_{1:n}; (S_0, S_1) \right) \right]  \leq -\epsilon n \right) &\leq \exp\left(- \frac{ \epsilon^2 n}{4} \right).
\end{align*}
We now calculate $\mathbb{E} \left[g\left(X_{1:n}, W_{1:n}; (S_0, S_1) \right) \right]$:
\begin{align*}
\mathbb{E} \left[g\left(X_{1:n}, W_{1:n}; (S_0, S_1) \right) \right] &= n \left[p \int_{t \in S_1} h_0(t) dt + (1-p) \int_{t \in S_0} h_1(t) dt \right]\\
&= n \left[p H_0(S_1) + (1-p) H_1(S_0) \right] \\
&= n \cdot q(S_0, S_1).
\end{align*}
By Assumption \ref{assumption:unique-minimizer}, the expectation has a unique minimizer $(S_0^{\star}, S_1^{\star}) \in \mathcal{M}_1$.

Observe that
\begin{align*}
\mathbb{P}\left( \Vert \hat{f}_n - f \Vert_2^2 > \delta\right) &= \mathbb{P} \left(D\left((S_0, S_1), (S_0^{\star}, S_1^{\star})\right) > \delta\right)\\
&= \mathbb{P} \left( (S_0, S_1) \not \in B_{\delta}(S_0^{\star}, S_1^{\star})\right)
\end{align*}
We therefore need to analyze the probability that there exists a monotone partition outside $B_{\delta}(S_0^{\star}, S_1^{\star}))$ with a smaller value of $g$ than $g\left(X_{1:n}, W_{1:n}; (S_0^{\star}, S_1^{\star}) \right)$. For all $(S_0, S_1) \in \mathcal{M}_1$, 
\begin{align*}
\mathbb{E} \left[g\left(X_{1:n}, W_{1:n}; (S_0, S_1)\right) \right] - \mathbb{E} \left[g\left(X_{1:n}, W_{1:n}; (S_0^{\star}, S_1^{\star}) \right) \right] &= n \left(q(S_0, S_1) - q\left(S_0^{\star}, S_1^{\star} \right) \right)
\end{align*}
We now use the concentration result with $\epsilon$ set to $\frac{1}{3} \left(q(S_0, S_1) - q\left(S_0^{\star}, S_1^{\star} \right) \right)$. For any $(S_0, S_1)$, with probability at least 
\[1 - \exp\left(- \frac{ \left(q(S_0, S_1) - q\left(S_0^{\star}, S_1^{\star} \right) \right)^2 n}{36} \right),\]
it holds that 
\[g\left(X_{1:n}, W_{1:n}; (S_0, S_1) \right) \geq  \mathbb{E} \left[g\left(X_{1:n}, W_{1:n}; (S_0, S_1) \right) \right] - \frac{n}{3}\left(q(S_0, S_1) - q\left(S_0^{\star}, S_1^{\star} \right) \right).\] 
Similarly, with the same probability, it holds that 
\[g\left(X_{1:n}, W_{1:n}; (S_0^{\star}, S_1^{\star}) \right) \leq  \mathbb{E} \left[g\left(X_{1:n}, W_{1:n}; (S_0^{\star}, S_1^{\star}) \right) \right] + \frac{n}{3}\left(q(S_0, S_1) - q\left(S_0^{\star}, S_1^{\star} \right) \right).\]
For a given $(S_0, S_1) \neq (S_0^{\star}, S_1^{\star})$, both of these events occur with probability at least \[1 - 2 \exp\left(- \frac{ \left(q(S_0, S_1) - q\left(S_0^{\star}, S_1^{\star} \right) \right)^2 n}{36} \right).\] In that case, 
\begin{align*}
&g\left(X_{1:n}, W_{1:n}; (S_0, S_1) \right)  - g\left(X_{1:n}, W_{1:n}; (S_0^{\star}, S_1^{\star}) \right) \\
&\geq \mathbb{E} \left[g\left(X_{1:n}, W_{1:n}; (S_0, S_1)\right) \right] - \frac{n}{3} \left(q(S_0, S_1) - q\left(S_0^{\star}, S_1^{\star} \right) \right) \\
&~~-  \mathbb{E} \left[g\left(X_{1:n}, W_{1:n}; (S_0^{\star}, S_1^{\star})\right) \right] - \frac{n}{3} \left(q(S_0, S_1) - q\left(S_0^{\star}, S_1^{\star}\right) \right)\\
&= n \left(q(S_0, S_1) - q\left(S_0^{\star}, S_1^{\star}\right) \right) - \frac{2n}{3} \left(q(S_0, S_1) - q\left(S_0^{\star}, S_1^{\star}\right) \right)\\
&= \frac{n}{3} \left(q(S_0, S_1) - q\left(S_0^{\star}, S_1^{\star}\right) \right).
\end{align*}
Therefore, in this situation, solution $(S_0, S_1)$ is suboptimal compared to solution $(S_0^{\star}, S_1^{\star})$. 

Observe that the cardinality of the set $\{g(X_{1:n}, W_{1:n}; (S_0, S_1)) : (S_0, S_1) \in \mathcal{M}_d\}$ is at most $n+1$. In other words, $g$ has at most $n+1$ possible values when we range over all possible monotone partitions.
Recall the definition of $q_{\text{min}}(\delta) = \min_{(S_0, S_1) \not \in B_{\delta}(S_0^{\star}, S_1^{\star})} q(S_0, S_1)$. 
By the previous analysis and the Union Bound, 
\begin{align*}
 \mathbb{P} \left( (S_0, S_1) \not \in B_{\delta}(S_0^{\star}, S_1^{\star})\right) &\leq (n+2) \exp\left(- \frac{ \left(q_{\text{min}}(\delta_1, \delta_2) - q\left(S_0^{\star}, S_1^{\star} \right) \right)^2 n}{36} \right)
\end{align*}
Therefore, with probability at least 
\begin{align*}
1- (n+2) \exp\left(- \frac{ \left(q_{\text{min}}(\delta_1, \delta_2) - q\left(S_0^{\star}, S_1^{\star} \right) \right)^2 n}{36} \right),
\end{align*}
it holds that $\Vert \hat{f}_n - f \Vert_2^2 \leq \delta$.


For $d \geq 2$ and $(S_0, S_1) \in \mathcal{M}_d$, let
\[g\left(X_{1:n}, W_{1:n}; (S_0, S_1) \right) = \sum_{i=1}^n \mathbbm{1}\left\{f(X_i + W_i) = 1, X_i  \in S_0 \right\} + \mathbbm{1}\left\{f(X_i + W_i) = 0, X_i \in S_1 \right\}.  \]
The function $g$ represents the error associated with partition $(S_0, S_1)$. Applying the McDiarmid inequality, 
\begin{align*}
\mathbb{P} \left( g\left(X_{1:n}, W_{1:n}; (S_0, S_1) \right) - \mathbb{E} \left[g\left(X_{1:n}, W_{1:n}; (S_0, S_1) \right) \right]  \geq \epsilon n \right) &\leq \exp\left(- \frac{ \epsilon^2 n}{4} \right)
\end{align*}
and 
\begin{align*}
\mathbb{P} \left( g\left(X_{1:n}, W_{1:n}; (S_0, S_1) \right) - \mathbb{E} \left[g\left(X_{1:n}, W_{1:n}; (S_0, S_1) \right) \right]  \leq -\epsilon n \right) &\leq \exp\left(- \frac{ \epsilon^2 n}{4} \right)
\end{align*}
Calculating the expectation,
\begin{align*}
\mathbb{E} \left[g\left(X_{1:n}, W_{1:n}; (S_0, S_1) \right) \right] &= n \left[p \int_{t \in S_1} h_0(t) dt + (1-p) \int_{t \in S_0} h_1(t) dt \right]\\
&= n \left[p H_0(S_1) + (1-p) H_1(S_0) \right] \\
&= n \cdot q(S_0, S_1).
\end{align*}
By Assumption \ref{assumption:unique-minimizer}, the function $q(S_0, S_1)$ has a unique minimizer, $(S_0^{\star}, S_1^{\star})$, that corresponds to the true function $f$. Therefore, if $\Vert \hat{f}_n - f \Vert_2^2$ is greater than $\delta$, then the function $\hat{f}_n$ must be outside of $B_{\delta}(S_0^{\star}, S_1^{\star})$. Then it must be the case that some $(S_0, S_1)$ outside of $B_{\delta}(S_0^{\star}, S_1^{\star})$ attained a lower value of $g$ than $g\left(X_{1:n}, W_{1:n}; (S_0^{\star}, S_1^{\star}) \right)$. We use concentration to upper bound the probability of this event.

First, we need to know how many possible objective values there are. This is upper bounded by the number of binary labelings of the set $\{X_1, \dots, X_n\}$. By Theorem \ref{thm:labeling-number-bounds}, it holds that 
\begin{align*}
\mathbb{E}[L(X_1, \dots X_n)] \leq \exp \left[ \left(2^s + 2\log(2) -1 \right) n^{\frac{s-1}{s}}\right].
\end{align*}
For any $\epsilon > 0$, the Markov inequality tells us that
\begin{align*}
\mathbb{P} \left(L(X_1, \dots X_n) \geq  t \right) &\leq \frac{\mathbb{E}[L(X_1, \dots X_n)]}{t}\\
&\leq \frac{ \exp \left[ \left(2^s + 2\log(2) -1 \right) n^{\frac{s-1}{s}}\right] }{t}.
\end{align*}
Setting $t =  \exp \left[n^{\frac{2s-1}{2s}}\right]$,
\begin{align*}
\mathbb{P} \left(L(X_1, \dots X_n) \geq  \exp \left[n^{\frac{2s-1}{2s}}\right] \right) &\leq \frac{ \exp \left[ \left(2^s + 2\log(2) -1 \right) n^{\frac{s-1}{s}}\right] }{\exp \left[ n^{\frac{2s-1}{2s}}\right]}.
\end{align*}
Therefore, with probability at least $1 - \frac{ \exp \left[ \left(2^s + 2\log(2) -1 \right) n^{\frac{s-1}{s}}\right] }{\exp \left[ n^{\frac{2s-1}{2s}}\right]}$, there are at most $ \exp \left[ n^{\frac{2s-1}{2s}}\right]$ labelings, and therefore function values. We bound the $L_2$ loss similarly to the proof for the case $d= s=1$, above. Recall that $q_{\text{min}}(\delta) =  \min_{(S_0, S_1) \not \in B_{\delta}(S_0^{\star}, S_1^{\star})} q(S_0, S_1)$. Set $\epsilon = \frac{1}{3} \left(q_{\text{min}}(\delta) - q(S_0^{\star}, S_1^{\star}) \right)$ in the McDiarmid bound so that the optimal value remains separated from the alternatives.
\begin{align*}
&\mathbb{P} \left( \Vert \hat{f} - f \Vert_2^2 > \delta \right) \\
&= \mathbb{P} \left( (S_0, S_1) \not \in   \left(B_{\delta}(S_0^{\star}, S_1^{\star}) \right)\right) \\
&\leq  \frac{ \exp \left[ \left(2^s + 2\log(2) -1 \right) n^{\frac{s-1}{s}}\right] }{\exp \left[n^{\frac{2s-1}{2s}}\right]} + \left(\exp \left[n^{\frac{2s-1}{2s}}\right] + 1\right) \exp\left(- \frac{ \left(q_{\text{min}}\left(\delta \right) - q\left(S_0^{\star}, S_1^{\star}\right) \right)^2 n}{36} \right)
\end{align*}
\end{proof}

\begin{proof}[Proof of Corollary \ref{corollary:noisy-input}]
We equivalently show that $s = o \left(\sqrt{\log(n)}\right)$ is sufficient. Analyzing the first term,
\begin{align*}
\exp \left\{n \left( n^{s \log_n(2)} + 2 \log(2) - 1 - n^{\frac{1}{2s}}\right) n^{-\frac{1}{s}} \right\} &\leq  \exp \left\{n^{1-\frac{1}{s}} \left( n^{s \log_n(2)} - n^{\frac{1}{2s}} + \frac{1}{2} \right)  \right\} \\
&=  \exp \left\{n^{1-\frac{1}{2s}} \left( n^{s \log_n(2) - \frac{1}{2s}} - 1 + \frac{1}{2}n^{-\frac{1}{2s}} \right)  \right\} \\
&\leq  \exp \left\{n \left( n^{s \log_n(2) - \frac{1}{2s}} - \frac{1}{2} \right)  \right\} \\
&=  \exp \left\{n \left( n^{\frac{1}{s} \left( s^2 \log_n(2) - \frac{1}{2}\right)} - \frac{1}{2} \right)  \right\} \\
&=  \exp \left\{n \left( n^{\frac{1}{s} \left( o(1)- \frac{1}{2}\right)} - \frac{1}{2} \right)  \right\} \\
&=  \exp \left\{n \left( n^{- \Theta(1) \frac{1}{s} } - \frac{1}{2} \right)  \right\} \\
&=  \exp \left\{n \left( n^{- \omega \left(\frac{1}{\sqrt{\log(n)}} \right)} - \frac{1}{2} \right)  \right\} \\
&=  \exp \left\{n \left( o \left(n^{-  \frac{1}{\sqrt{\log(n)}} }\right) - \frac{1}{2} \right)  \right\} \\
&=  \exp \left\{n \left( o \left(n^{-  \frac{1}{\log(n)} }\right) - \frac{1}{2} \right)  \right\} \\
&=  \exp \left\{n \left( o \left(n^{-  \frac{\log_n(2)}{\log(2)} }\right) - \frac{1}{2} \right)  \right\} \\
&=  \exp \left\{n \left( o \left(2^{-  \frac{1}{\log(2)} }\right) - \frac{1}{2} \right)  \right\} \\
&=  \exp \left\{n \left( o \left(1\right) - \frac{1}{2} \right)  \right\} \\
&=  \exp \left\{-\Theta(1)n  \right\}
\end{align*}
We have assumed that the expression $\left(q_{\text{min}}\left(\delta \right) - q\left(S_0^{\star}, S_1^{\star}\right) \right)^2$ is constant in $s$. Analyzing the second term,
\begin{align*}
\exp \left[n^{\frac{2s-1}{2s}}\right] \exp\left(- \frac{ \left(q_{\text{min}}\left(\delta \right) - q\left(S_0^{\star}, S_1^{\star}\right) \right)^2 n}{36} \right) &= \exp \left \{n \left(n^{-\frac{1}{2s}} - \Theta(1)\right) \right \}\\
&= \exp \left \{n \left(n^{-\frac{1}{2 o \left(\sqrt{\log(n)}\right)}} - \Theta(1)\right) \right \}\\
&= \exp \left \{n \left(o \left(n^{-\frac{1}{2 \sqrt{\log(n)}}}\right) - \Theta(1)\right) \right \}\\
&= \exp \left \{n \left(o(1) - \Theta(1)\right) \right \}\\
&=  \exp \left\{-\Theta(1)n  \right\}
\end{align*}
\end{proof}

We provide an analogue of Theorem \ref{thm:consistency-binary} in the sparse setting ($s < d$). First we need some definitions, similar to those that precede Theorem \ref{thm:consistency-binary}. We write $x =_A y$ if $x \preceq_A y$ and $x \succeq_A y$.
\begin{definition}[$s$-Sparse Monotone Partition]
We say that $(S_0, S_1)$ is an \emph{$s$-sparse monotone partition} of $\mathbb{R}^d$ if
\begin{enumerate}
\item $S_0$ and $S_1$ form a partition of $\mathbb{R}^d$. That is, $S_0 \cup S_1 = \mathbb{R}^d$ and $S_0 \cap S_1 = \emptyset$.
\item There exists a set $A \subset [d]$ such that for all $x, y \in \mathbb{R}^d$, if $x \preceq_{A} y$, then either (i) $x,y \in S_0$, (ii) $x, y \in S_1$, or (iii) $x \in S_0, y \in S_1$. Note that this implies that if $x =_{A} y$, then either $x,y \in S_0$ or $x,y \in S_1$.
\end{enumerate}
Let $\mathcal{M}_{s,d}$ be the set of all $s$-sparse monotone partitions of $\mathbb{R}^d$.
\end{definition}
Note that there is a one-to-one correspondence between monotone partitions and $s$-sparse binary coordinate-wise monotone functions.

Let $Y = f(X+W)$ represent our model, with $d < s$, and with $f$ corresponding to an $s$-sparse monotone partition $(S_0^{\star}, S_1^{\star})$. That is, $f(x) = 0$ for $x \in S_0^{\star}$ and $f(x) = 1$ for $x \in S_1^{\star}$.  Let $h_0(x)$ be the probability density function of $X$, conditional on $Y = 0$. Similarly, let $h_1(x)$ be the probability density function of $X$, conditional on $Y = 1$. For $(S_0, S_1) \in \mathcal{M}_{s,d}$, let 
\begin{align*}
H_0(S_1) = \int_{z \in S_1} h_0(z) dz ~\text{ and } ~H_1(S_0) = \int_{z \in S_0} h_1(z) dz.
\end{align*}
Finally, let $p$ be the probability that $Y = 0$. Let 
\[q(S_0, S_1) \triangleq p H_0(S_1) + (1-p) H_1(S_0).\]
The value of $q(S_0, S_1)$ is the probability of misclassification, under the $s$-sparse monotone partition $(S_0, S_1)$. 

\begin{assumption}\label{assumption:unique-minimizer-sparse}
We assume that $q$ has a unique minimizer on $\mathcal{M}_{s,d}$, which is $(S_0^{\star}, S_1^{\star})$. 
\end{assumption}

\begin{definition}[Discrepancy]
For two $s$-sparse monotone partitions $(S_0, S_1)$ and $(S_0', S_1')$, the discrepancy function $D: \mathcal{M}_{s,d} \times \mathcal{M}_{s,d} \to [0,1]$ is defined as follows.
\[D\left((S_0, S_1), (S_0', S_1')\right) \triangleq \mathbb{P}\left(X \in S_0 \cap S_1' \right) + \mathbb{P}\left(X \in S_0' \cap S_1\right)  \]
Also let
\[B_{\delta}^s \left(S_0^{\star}, S_1^{\star}\right) \triangleq \{(S_0, S_1) \in \mathcal{M}_{s,d} :  D\left((S_0, S_1), (S_0^{\star}, S_1^{\star})\right)\leq \delta \}\] 
be the set of $s$-sparse monotone partitions with discrepancy at most $\delta$ from $(S_0^{\star}, S_1^{\star})$.
\end{definition}

\begin{theorem}\label{thm:noisy-input-sparse}
Suppose Assumption \ref{assumption:unique-minimizer-sparse} holds, and the components of $W$ are independent. Let $\hat{f}_n$ be the estimator derived from Algorithm \ref{alg:sequential-recovery}, and let 
\[q_{\text{min}}(\delta) \triangleq \min \left\{q(S_0, S_1): (S_0, S_1) \not \in B_{\delta}^s(S_0^{\star}, S_1^{\star})\right\} > q(S_0^{\star}, S_1^{\star}).\] 
Then for any $0 < \delta \leq 1$, 
\begin{align*}
&\mathbb{P} \left( \Vert \hat{f} - f \Vert_2 >\delta \right) \leq\\
&\frac{ \exp \left[ \left(2^s + 2\log(2) -1 \right) n^{\frac{s-1}{s}}\right] }{\exp \left[n^{\frac{2s-1}{2s}}\right]} +\left(\binom{d}{s}\exp \left[n^{\frac{2s-1}{2s}}\right] + 1\right) \exp\left(- \frac{ \left(q_{\text{min}}\left(\delta\right) - q\left(S_0^{\star}, S_1^{\star} \right) \right)^2 n}{36} \right).
\end{align*}
\end{theorem}

\begin{proof}
The proof is analogous to the proof of Theorem \ref{thm:consistency-binary}, with the above definition for the function $q$. Recall that in the proof of Theorem \ref{thm:consistency-binary}, we needed to upper bound the number of possible function values. Here, the number of possible function values is upper bounded by the number of $s$-sparse binary labelings, which are those labelings corresponding to $s$-sparse monotone partitions. Let $L_s(X_1, \dots X_n)$ be the number of $s$-sparse binary labelings.

By Theorem \ref{thm:labeling-number-bounds}, it holds that 
\begin{align*}
\mathbb{E}[L_s(X_1, \dots X_n)] \leq \binom{d}{s}\exp \left[ \left(2^s + 2\log(2) -1 \right) n^{\frac{s-1}{s}}\right].
\end{align*}
For any $\epsilon > 0$, the Markov inequality tells us that
\begin{align*}
\mathbb{P} \left(L_s(X_1, \dots X_n) \geq  t \right) &\leq \frac{\mathbb{E}[L_s(X_1, \dots X_n)]}{t}\\
&\leq \frac{ \binom{d}{s}\exp \left[ \left(2^s + 2\log(2) -1 \right) n^{\frac{s-1}{s}}\right] }{t}.
\end{align*}
Setting $t =  \binom{d}{s} \exp \left[n^{\frac{2s-1}{2s} }\right]$,
\begin{align*}
\mathbb{P} \left(L_s(X_1, \dots X_n) \geq  \binom{d}{s} \exp \left[n^{\frac{2s-1}{2s}}\right] \right) &\leq \frac{ \exp \left[ \left(2^s + 2\log(2) -1 \right) n^{\frac{s-1}{s}}\right] }{\exp \left[ n^{\frac{2s-1}{2s}}\right]}.
\end{align*}
Therefore, with probability at least $1 - \frac{ \exp \left[ \left(2^s + 2\log(2) -1 \right) n^{\frac{s-1}{s}}\right] }{\exp \left[ n^{\frac{2s-1}{2s}}\right]}$, there are at most $\binom{d}{s} \exp \left[ n^{\frac{2s-1}{2s}}\right]$ $s$-sparse binary labelings, and therefore function values.
\end{proof}

\begin{proof}[Proof of Corollary \ref{corollary:noisy-input-simultaneous}]
We have assumed that $s$ is constant. For fixed $(S_0, S_1)$, the value of $q(S_0, S_1)$ does not change if we increase the overall dimension, because of the uniformity of $X$ and the independence of the coordinates of $W$. Therefore, $q$ does not depend on $d$ when $s$ is fixed, and so $q_{\text{min}}\left(\delta\right) - q\left(S_0^{\star}, S_1^{\star} \right) = \Theta(1)$. We now analyze the bound in Theorem \ref{thm:noisy-input-sparse}. Since $s$ is constant, the first term goes to zero. Analyzing the second term,
\begin{align*}
&\left(\binom{d}{s}\exp \left[n^{\frac{2s-1}{2s}}\right] + 1\right) \exp\left(- \frac{ \left(q_{\text{min}}\left(\delta\right) - q\left(S_0^{\star}, S_1^{\star} \right) \right)^2 n}{36} \right)\\
&\leq  \left(\exp \left[s \log(d) + n^{\frac{2s-1}{2s}}\right]+ 1\right) \exp\left(- \frac{ \left(q_{\text{min}}\left(\delta\right) - q\left(S_0^{\star}, S_1^{\star} \right) \right)^2 n}{36} \right)\\
&=  \exp \left[s \log(d) + n^{\frac{2s-1}{2s}} - \Theta(1)n \right] + e^{-\Theta(1) n}
\end{align*}
If $n = \omega(\log(d))$, the second term goes to zero.
\end{proof}

The proof of Theorem \ref{thm:support-recovery-noisy-input-multiple} requires Lemmas \ref{lemma:support-recovery-noisy-input-probability} and \ref{lemma:support-recovery-noisy-input}.
\begin{lemma}\label{lemma:support-recovery-noisy-input-probability}
For all $k \in A$, it holds that $\overline{p}_{k} > 0$. 
\end{lemma}
\begin{proof}
We need to show that
\begin{align*}
\mathbb{P} \left( Y_1 = 1, Y_2 = 0 | X_{1,k} > X_{2,k} \right) > \mathbb{P} \left( Y_1 = 0, Y_2 = 1 | X_{1,k} > X_{2,k}  \right).
\end{align*}
The proof is similar to the proof of Lemma \ref{lemma:stochastically-increasing}.
Consider the following procedure. We sample $X_1$ and $X_2$ independently and uniformly on $[0,1]^d$. Fix $k \in A$. Let
\begin{align*}
X_+ = \begin{cases}
X_1 & \text{ if } X_{1,k} > X_{2,k}\\
X_2 &\text{ otherwise}
\end{cases}
\end{align*}
and
\begin{align*}
X_- = \begin{cases}
X_1 & \text{ if } X_{1,k} \leq X_{2,k}\\
X_2 &\text{ otherwise.}
\end{cases}
\end{align*}
In other words, $X_+$ is the right point according to coordinate $k$ and $X_-$ is the left point. As in the proof of Lemma \ref{lemma:stochastically-increasing}, we can equivalently define $\overline{p}_k$ as
\[\overline{p}_{k} =\mathbb{P} \left( f(X_+ + W_1)  > f(X_- + W_2) \right) - \mathbb{P} \left( f(X_+ + W_1) < f(X_- + W_2) \right)  .\]
Therefore, our goal is to show that
\begin{align*}
&\mathbb{P} \left( f(X_+ + W_1)  > f(X_- + W_2) \right) > \mathbb{P} \left( f(X_+ + W_1) < f(X_- + W_2) \right).
\end{align*}
Due to the monotonicity of $f$ with respect to $A \ni k$, it holds that
\begin{align*}
&\mathbb{P} \left( f(X_+)  > f(X_-) \right) > \mathbb{P} \left( f(X_+)  < f(X_-) \right).
\end{align*}
Coupling $W_1$ and $W_2$ across the events $\{f(X_+)  > f(X_-)\}$ and $\{f(X_+)  < f(X_-)\}$, we obtain 
\begin{align*}
&\mathbb{P} \left( f(X_+ + W_1)  > f(X_- + W_2) \right) > \mathbb{P} \left( f(X_+ + W_1) < f(X_- + W_2) \right).
\end{align*}
\end{proof}

\begin{lemma}\label{lemma:support-recovery-noisy-input}
Let $s=1$, and assume without loss of generality that $A = \{1\}$. Algorithm \ref{alg:sequential-recovery} (which is equivalent to Algorithm \ref{alg:simultaneous-recovery} when $s=1$) recovers the correct coordinate with probability at least
\[1- (d-1) \exp\left(-\frac{n \overline{p}_{1}^2}{16} \right).\]
\end{lemma}
\begin{proof}
The proof is identical to the proof of Lemma \ref{lemma:one-coordinate}, with $p$ replaced by $\overline{p}$. Lemma \ref{lemma:support-recovery-noisy-input-probability} guarantees that $\overline{p}_k > 0$ for all $k \in A$.
\end{proof}

\begin{proof}[Proof of Theorem \ref{thm:support-recovery-noisy-input-multiple}]
The proof is nearly identical to the proof of Theorem \ref{thm:multiple-coordinates}, and relies on Lemma \ref{lemma:support-recovery-noisy-input}.
\end{proof}
\begin{proof}[Proof of Corollary \ref{corollary:noisy-input-support-recovery}]
The proof is identical to the proof of Corollary \ref{corollary:recovery}, with $p$ replaced by $\overline{p}$.
\end{proof}

\begin{proof}[Proof of Corollary \ref{corollary:two-stage-noisy-input}]
Support recovery fails with probability at most \[(d-s)\sum_{k=1}^s (d+1 -k) \exp\left(-\frac{n \overline{p}_{k}^2}{16} \right).\] If it succeeds, the probability of the $L_2$ norm error exceeding $\delta$ is upper bounded by the value in Theorem \ref{thm:consistency-binary}.
 Then $\mathbb{P} \left( \Vert \hat{f}_n - f \Vert_2^2 > \delta \right)$ is at most
\begin{align*}
&(d-s)\sum_{k=1}^s (d+1 -k) \exp\left(-\frac{n \overline{p}_{k}^2}{16} \right)\\
&+ \frac{ \exp \left[ \left(2^s + 2\log(2) -1 \right) n^{\frac{s-1}{s}}\right] }{\exp \left[n^{\frac{s-1}{s} + \epsilon}\right]} +\left(\exp \left[n^{\frac{s-1}{s} + \epsilon}\right] + 1\right) \exp\left(- \frac{ \left(q_{\text{min}}\left(\delta\right) - q\left(S_0^{\star}, S_1^{\star} \right) \right)^2 n}{36} \right).
\end{align*}
Therefore, if $n = \omega(\log(d))$ and $n = e^{\omega(s^2)}$, the estimator is consistent under the assumptions of Corollary \ref{corollary:two-stage-noisy-input}.
\end{proof}

\end{document}